\documentclass[12pt]{amsart}
\usepackage{a4wide}
\usepackage{amssymb}
\usepackage{amsthm}
\usepackage{amsmath}
\usepackage{amscd}

\usepackage{verbatim}
\usepackage[all]{xy}

\numberwithin{equation}{section}

\theoremstyle{plain}
\newtheorem{theorem}{Theorem}[section]
\newtheorem{corollary}[theorem]{Corollary}
\newtheorem{lemma}[theorem]{Lemma}
\newtheorem{proposition}[theorem]{Proposition}

\newtheorem{conjecture}[theorem]{Conjecture}
\theoremstyle{definition}

\newtheorem{remark}[theorem]{Remark}

\theoremstyle{remark}

\newcommand{\OO}{\mathcal O}
\newcommand{\A}{\mathbb{A}}
\newcommand{\R}{\mathbb{R}}
\newcommand{\G}{\mathbb{G}}
\newcommand{\Q}{\mathbb{Q}}
\newcommand{\Z}{\mathbb{Z}}

\newcommand{\C}{\mathbb{C}}

\renewcommand{\H}{\mathbb{H}}

\newcommand{\D}{\mathbb{D}}

\newcommand{\zxz}[4]{\begin{pmatrix} #1 & #2 \\ #3 & #4 \end{pmatrix}}
\newcommand{\abcd}{\zxz{a}{b}{c}{d}}

\newcommand{\kzxz}[4]{\left(\begin{smallmatrix} #1 & #2 \\ #3 & #4\end{smallmatrix}\right) }

\newcommand{\calF}{\mathcal{F}}

\newcommand{\eps}{\varepsilon}

\newcommand{\norm}{\operatorname{N}}

\newcommand{\tr}{\operatorname{tr}}

\newcommand{\Cl}{\operatorname{Cl}}

\newcommand{\Pet}{\text{\rm Pet}}
\newcommand{\Gr}{\operatorname{Gr}}
\newcommand{\GL}{\operatorname{GL}}
\newcommand{\SO}{\operatorname{SO}}

\newcommand{\Res}{\operatorname{Res}}
\newcommand{\Gal}{\operatorname{Gal}}

\newcommand{\cha}{\operatorname{{Char}}}

\newcommand{\diag}{\operatorname{diag}}
\newcommand{\ord}{\operatorname{ord}}
\newcommand{\Gspin}{\operatorname{GSpin}}

\newcommand{\ff}{\hbox{if }}
\newcommand{\SL}{\operatorname{SL}}
\newcommand{\Diff}{\operatorname{Diff}}

\newcommand{\kk}{\operatorname{E}}

\newcommand{\ph}{\phi}

\newcommand{\Ind}{\operatorname{Ind}}
\newcommand{\IM}{\operatorname{IM}}

\begin{document}

\title[Difference of modular functions and their CM value factorization
]{Difference of modular functions and their CM value factorization \\ to appear in Trans.  AMS}

\author[ Tonghai Yang and Hongbo Yin]{Tonghai Yang and Hongbo Yin}

\address{Department of Mathematics, University of Wisconsin Madison, Van Vleck Hall, Madison, WI 53706, USA}
\email{thyang@math.wisc.edu}
\address{Academy of Mathematics and Systems Science, Morningside center of Mathematics, Chinese Academy of Sciences, Beijing 100190}
\email{yinhb@math.ac.cn}


\thanks{The first author is partially supported by a NSF grant DMS-1500743.}


\begin{abstract} In this paper, we use Borcherds lifting and  the big CM value formula of Bruinier, Kudla, and Yang to give an explicit factorization formula for the norm of $\Psi(\frac{d_1+\sqrt{d_1}}2) -\Psi(\frac{d_2+\sqrt{d_2}}2)$, where $\Psi$ is the $j$-invariant or  the Weber invariant $\omega_2$.  The $j$-invariant case gives another proof of the well-known Gross-Zagier factorization formula of singular moduli,  while the Weber invariant case gives a proof of the Yui-Zagier conjecture for $\omega_2$.  The method used here could be extended to deal with other modular functions on a genus zero modular curve.
\end{abstract}

\maketitle

\setcounter{tocdepth}{1}
\tableofcontents

\section{Introduction} \label{intro}

In 1980s, Gross and Zagier discovered a beautiful factorization formula for the singular moduli  \cite{GZSingular} in preparation to their  well-known Gross-Zagier formula. It was extended slightly by Dorman \cite{Dorman}  which can be stated as follows (see Remark \ref{rem:GZformula}).

\begin{theorem}\label{theo1.1} (Gross-Zagier, Dorman) \label{theo:GrossZagier} Let $E_i=\Q(\sqrt d_i)$ be two imaginary quadratic fields of fundamental discriminants $d_i$ with $(d_1, d_2) =1$, let $F=\Q(\sqrt D)$ with $D=d_1d_2$ and $E=\Q(\sqrt{d_1}, \sqrt{d_2})$. Let $j(\tau)$ be the well-known $j$-invariant. Then
$$
\sum_{[\mathfrak a_i] \in \Cl(E_i)} \log|j(\tau_{\mathfrak a_1} ) -j(\tau_{\mathfrak a_2})|^{\frac{8}{w_1 w_2}}
 =\sum_{\substack{t =\frac{m+\sqrt D}2 \in \OO_F \\ |m|< \sqrt D} }
   \sum_{\mathfrak p \hbox{ inert in } E/F} \frac{1+\ord_\mathfrak p(t\OO_F)}2  \rho(t \mathfrak p^{-1} ) \log (\norm(\mathfrak p)).
$$
Here $w_i$ is the number of roots of unity in $E_i$, and  for an integral ideal $\mathfrak a$ of $F$
$$
\rho(\mathfrak a) = |\{\mathfrak A \subset \OO_E: \norm_{E/F} (\mathfrak A) =\mathfrak a\}|.
$$
Finally, for an integral ideal $\mathfrak a_i$ of $E_i$ with
$$
\mathfrak a_i = \Z a_i + \Z \frac{b_i+\sqrt{d_i}}2, \quad a_i =\norm(\mathfrak a_i),
$$
its associated CM point is $\tau_{\mathfrak a_i} = \frac{b_i +\sqrt{d_i}}{2a_i}$.
\end{theorem}

This gives a beautiful factorization formula for $\norm(j(\frac{d_1+\sqrt{d_1}}2) -j(\frac{d_2+\sqrt{d_2}}2))$ (up to sign). In particular, the biggest prime factor of this norm is less than or equal to $D/4$, extremely small comparing to the norm. The first few examples of this phenomenon were discovered by Berwick in 1920s \cite{Berwick27}. For example, one has
\begin{align*}
j(\frac{1+  \sqrt{-163}}2)-j(\frac{1+\sqrt{-3}}2))
   &=-2^{18} 3^3 5^3 23^3 29^3=-262537412640768000,
\\
j(\frac{1+  \sqrt{-163}}2)-j(i) &=-2^6 3^6 7^2 11^2 19^2 127^2 163=-262537412640769728.
\end{align*}

In  1997, Yui and Zagier  \cite{YuiZagier} defined a  mysterious CM value $\bold f(\frac{d+\sqrt d}2)$ via the three Weber functions of level $48$
(when $d \equiv 1 \pmod 8$ and $3 \nmid d$) and proved that it is defined over the Hilbert class field of $\Q(\sqrt d)$, and claimed that its Galois conjugates are the CM values at other CM points of the same discriminant  $d$ with some modifications, which was later proved by Alice Gee using Shimura's reciprocity law. In addition,  Yui and Zagier  gave a conjectural factorization formula for the norm of $\bold f(\frac{d_1+\sqrt{d_1}}2)^a - \bold f(\frac{d_2+\sqrt{d_2}}2)^a$ similar to the Gross-Zagier factorization formula for any positive integer $a|24$.  For example, when $a=24$, the  conjecture can be restated as follows.

\begin{conjecture}\label{conj:YuiZagier} Let the notation be as in Theorem  \ref{theo1.1}, and assumer further $d_1 \equiv d_2 \equiv 1 \pmod 8$. Let
$$
\omega_2(\tau)=2^{12} q\prod_{n>0}(1+q^{n})^{24}=2^{12}\cdot \frac{\Delta(2\tau)}{\Delta(\tau)}
$$
be the Weber  modular function for $\Gamma_0(2)$. Then
$$
\sum_{[\mathfrak a_i|\in \Cl(E_i)} \log|\omega_2(\tau_{\mathfrak a_1})-\omega_2(\tau_{\mathfrak a_2})|^2
 =\sum_{\substack{t =\frac{m+\sqrt D}2\\ |m|< \sqrt D, \hbox{\tiny odd} \\ m^2 \equiv D \pmod {16}} }
   \sum_{\mathfrak p \hbox{ inert in } E/F} \frac{1+\ord_\mathfrak p(t\OO_F)}2  \rho(t \mathfrak p^{-1}  \mathfrak p_t^{-2}) \log (\norm(\mathfrak p)).
$$
Here $\mathfrak p_t$ is the unique prime ideal of $F$ above $2$ such that $\ord_{\mathfrak p_t}(t\OO_F) \ge 1$, and  for each ideal class $[\mathfrak a_i] \in \Cl(E_i)$, we choose a representative $\mathfrak a_i$ integral with norm prime to $2$,  i.e.,
$$
\mathfrak a_i=\Z a_i + \Z \frac{b_i+ \sqrt{d_i}}2, \quad \hbox{ with }  2 \nmid a_i,  \quad a_i >0. \quad \tau_{\mathfrak a_i} = \frac{b_i+ \sqrt{d_i}}{2a_i}.
$$

\end{conjecture}
They  provided some numerical evidence in their paper. Notice that the biggest prime factor of this norm is less than or equal to $D/16$.  In  this paper, we will prove  this conjecture.

\begin{theorem}  \label{theo1.3} Conjecture \ref{conj:YuiZagier} is true.

\end{theorem}

In his 2006 thesis \cite{Schofer}, Schofer used regularized theta lifting to generalize the Gross-Zagier factorization formula to small CM values of the so-called Borcherds products on the orthogonal Shimura varieties of type $(n, 2)$. Bruinier and Yang generalized it to big CM values of Hilbert  modular forms (which are Borcherds products) over a real quadratic field \cite{BY06}. More recently,  Bruinier, Kudla, and Yang (\cite{BKY12})  generalized it to big CM values of Borcherds products on Shimura varieties of orthogonal type $(n, 2)$, following Schofer and \cite{BY09}. On a different track,  Lauter, Goren, and Viray  have used geometric methods to generalize the Gross-Zagier formula to Igusa's $j$-invariants  for genus two curves, which have important applications to genus two curve cryptosystem (\cite{GL07}, \cite{GL12}, \cite{LV15}). Yang also proved Lauter's conjecture on Igusa's $j$-invariants  by combining  the result in \cite{BY06} with his work on arithmetic intersection \cite{Yang13}. The big CM value formula in \cite{BY06}, \cite{BKY12} has also been used to prove certain cases of the Colmez conjecture (\cite{Yang10-AJM},  \cite{Yang10-CJM},  \cite{Yang13}, \cite{BHKRY}), and the average Colmez conjecture (\cite{AGHMP}). Dongxi Ye is extending the result to other modular curves of genus zero \cite{Ye}.

This paper is the first part of our effort to prove Yui and Zagier's conjectural formula  using the big CM value  formula.

     The general idea is as follows. Let $\Gamma$ be a congruence subgroup such that  the compactification of $X_\Gamma=\Gamma\backslash \H$ has genus zero, and let $\Psi$ be a generator of the function field of $X_\Gamma$, which is a modular function for $\Gamma$.  Then  the difference function $\Psi(z_1) -\Psi(z_2)$ is a two variable modular function on $X_\Gamma \times  X_\Gamma$ with divisor being the diagonal divisor. We view $X_\Gamma \times X_\Gamma$ as an orthogonal Shimura variety of type $(2,2)$ associated to $(V=M_2(\Q), Q=N\det)$ for some positive integer $N$. One can show that the  diagonal divisor is a special divisor on the product $X_\Gamma \times X_\Gamma$ so that $\Psi(z_1) -\Psi(z_2)$ has a chance to be a Borcherds lifting (product) of some weakly holomorphic modular forms (\cite{Borcherds98}, \cite{Bruinier02}).
     The first task is to find a weakly holomorphic modular  form, if any,  whose  Borcherds lifting  is the  difference $\Psi(z_1) -\Psi(z_2)$  (\cite{Borcherds98}, \cite{Bruinier02}, see Section \ref{sect:Borcherds}). There are two complications even with Bruinier's converse results (\cite{Bruinier02}, \cite{Bruinier14}). First, when $N >1$, Bruinier's converse theorem does not apply.  Second, there are two variable modular functions whose divisors are only supported on the boundary, so it is not enough to compare the divisors of the Borcherds product with  our function only in the open Shimura variety.  We also need to understand their boundary behavior. The Borcherds product expansion  is important in this aspect.  In this paper, we are only successful in this step for the Weber functions $\omega_i$ (Section \ref{sect:Weber})  but not  for the more interesting Weber functions $\bold f_i$ of level $48$.

The second task is to identify a pair of Heegner points $(\tau_{\mathfrak a_1}, \tau_{\mathfrak a_2})$ with a big CM  point on $X_\Gamma \times  X_\Gamma$  associated to the CM  number field $E=\Q(\sqrt{d_1}, \sqrt{d_2})$ in the sense of Bruinier, Kudla, and Yang in \cite{BKY12}. This is done in Section \ref{sect:BigCM}. The third task is to apply the big CM value formula in \cite{BKY12} (assuming that $\Psi$ is a Borcherds lifting) to provide the expected formula.  One serious problem (for the Yui-Zagier conjecture) is that the big CM cycle in \cite{BKY12} is likely bigger in size than the ideal class groups used in Yui-Zagier's conjectural formula in general. One might need to use Shimura's reciprocity law to analyze the Galois action on the values as in  \cite{Gee} to solve the problem.  In the case of $\omega_2$, the condition  $d_i \equiv 1 \pmod 8$ allows us to choose an embedding from $E_i$ to $\GL_2(\Q)$ so that the ideal class group  maps into $X_0(2)$ nicely.  Another minor complication (interesting feature) is the explicit computation of the Fourier coefficient of the derivatives of some incoherent Eisenstein series since Schwartz functions are not  factorizable  in the $\omega_2$ case (Section \ref{sect:Weber}).

 Here is the organization of this paper. In Section \ref{sect:Borcherds}, we review Borcherds lifting, Borcherds product expansion (\cite{Borcherds98}, \cite{Bruinier02}), and the big CM value formula (\cite{BKY12}). In Section \ref{sect:BigCM}, we identify the product $X_\Gamma \times X_\Gamma$  of two copies of a modular curve
with a Shimura variety of orthogonal type $(2, 2)$ and identify its big CM  points with pairs of the CM  points on the modular curve $X_\Gamma$. In Section \ref{GrossZagier}, we reprove Theorem \ref{theo:GrossZagier} using the big CM value formula. In Section \ref{sect:Weber}, we first identify $\omega_2(z_1) -\omega_2(z_2)$
 with a Borcherds lifting of some explicit weakly holomorphic modular forms, and then use the big CM value formula to prove Theorem \ref{theo1.3}.

{\bf Acknowledgement}:  The authors thank Dongxi Ye for carefully reading the earlier drafts and giving us very valuable feedback and correction. They thank the anonymous referee for his/her suggestion and comments. Part of the work was done when Yin visited  Department of Mathematics at UW-Madison during 2014-15  and when Yang visited the Morninside Center of Mathematics at Beijing during summer of 2015. The authors thank both institutes for providing excellent working conditions for them.

\section{Borcherds lifting and the Big CM value formula} \label{sect:Borcherds}

\subsection{Borcherds lifting and Borcherds product expansion} \label{sect4.1} \label{Borcherds1}

In this subsection, we review the beautiful work of Borcherds in details using slightly different convention and notation for our purpose. Let $(V, Q)$ be a quadratic space over $\Q$ of signature $(n, 2)$, and let $L$ be an even integral lattice, i.e., $Q(x)=\frac12(x,x) \in
\mathbb Z$ for $x \in L$. Let
$$
L'=\{ y \in V:\, (x, y) \in \mathbb Z, \hbox{ for } x \in L \}
\supset L
$$
be its dual. We assume in this paper that $n$ is even for simplicity.  Let $H=\Gspin(V)$, and let  $\mathbb D$ be the oriented negative $2$-planes in $V_\R$. Then  for a compact open subgroup $K$ of $H(\A_f)$, there is a Shimura variety $X_K$ defined over $\Q$ such that
$$
X_K(\C) = H(\Q) \backslash ( \mathbb D \times  H(\A_f)/K).
$$
We will identify $X_K$ with $X_K(\C)$ in this section. We  assume that $K$
 fixes $L$ and acts on $L'/L$ trivially. The Hermitian symmetric domain $\mathbb  D$ has two other useful forms.
 Let
 \begin{equation}
 \mathcal L = \{ w \in V_\C:\,  (w, w) =0,  \quad (w, \bar w)<0\}.
 \end{equation}
 Then one has an isomorphism
 $$
 \mathcal L/\C^\times \cong \mathbb D, \quad  w=u + i v \mapsto \R u + \R (-v).
 $$
 This isomorphism gives a complex structure on  $\mathbb D$, and we can view  $\mathcal L$ as a line bundle over $\mathbb D$---the tautological line bundle. It descends to a line bundle $\mathcal L_K$ over $X_K$---the line bundle of modular forms of weight $1$ on $X_K$.  Finally, given an isotropic element $\ell \in V$, choose another element $\ell' \in V$ such that $(\ell, \ell')=1$, and let $V_0=(\Q\ell + \Q \ell')^\perp$. Then we have a tube domain (associated to $(\ell, \ell')$):
$$
\mathcal H =H_{\ell, \ell'}=\{ z =x +i y \in  V_{0,\C}:\,  Q(y) <0\}.
$$
The map
$$
w=w_{\ell, \ell'}: \mathcal H \rightarrow \mathcal L, \quad  w(z) = z + \ell' - (Q(z) + Q(\ell'))\ell
$$
 gives an isomorphism $\mathcal H_{\ell, \ell'} \cong \mathcal L/\C^\times$, and actually a nowhere vanishing section of the line bundle $\mathcal L$.  We emphasize that $w$ depends on the choice  of the primitive isotropic vector $\ell$ and the subspace $\Q \ell +\Q \ell'$, but not $\ell'$. Furthermore, this map $w$ induces an action of $\Gamma =K\cap H(\Q)^+$ on $\mathcal H$ and an automorphy factor $j(\gamma, z)$  characterized by the following identity:
  \begin{equation}
  \gamma w(z) = \nu(\gamma) j(\gamma, z) w(\gamma z).
  \end{equation}
Here $H(\R)^+$ is the identity component of $H(\R)$, and $H(\Q)^+=H(\Q)\cap H(\R)^+$, $\nu(\gamma)$ is the spinor norm of $\gamma$. This action preserves the two connected components of $\mathcal H=\mathcal H^+ \cup \mathcal H^-$. A  (meromorphic) function $\Psi$ on $\mathcal H^+$ is called a (meromorphic) modular form for $\Gamma$ of weight $k$ if
\begin{equation}
\Psi(\gamma z) = j(\gamma, z)^k \Psi(z).
\end{equation}
Alternatively, it is a section of  the line bundle $\mathcal L_K^k$ over $X_K$.

For a vector $x\in V$ with $Q(x)>0$, and $h\in H(\A_f)$. Let
$$
H_x =\{ h \in H:\,  h(x) =x\},  \quad \mathbb D_x =\{ z \in \mathbb D:\, (x, z) =0\}, \quad \hbox{ and } \quad K_{x, h} =H_x(\A_f) \cap hKh^{-1}.
$$
Then the map
$$
 H_x(\Q) \backslash (\mathbb D_x \times H_x(\A_f)/K_{x, h}) \rightarrow X_K(\C), \quad [z, h_1] \mapsto [z, h_1h]
$$
gives a divisor $Z(x, h)$ in $X_K$. It is actually defined over $\Q$.  For a rational number $m>0$ and $\phi\in S(V_f)$, if there is an $x \in V$ with $Q(x)=m$, we define, following Kudla \cite{Kudla97}, the weighted special divisor
$$
Z(m, \phi) =\sum_{ h\in H_x(\A_f) \backslash H(\A_f)/K} \phi(h^{-1}x) Z(x, h).
$$
When there is no $x\in V$ with $Q(x)=m$, we simply set $Z(m, \phi)=0$.

 Associated to the quadratic space $V$ is a reductive dual pair $(\SL_2, O(V))$ and a Weil  representation $\omega =\omega_{V, \psi}$ of $\SL_2(\A)$ on $S(V_\A)=S(V_f) \otimes S(V_\infty)$, where $V_f =V\otimes_\Q \A_f$ and  $V_\infty =V\otimes_\Q \Q_\infty =V\otimes_\Q \R$. Embed $\SL_2(\Z)$ into $\SL_2(\hat\Z)$ diagonally, and let $S_L \subset S(\A_f)$ be the subspace of Schwartz functions $\phi$ which is supported on $\hat L'=L'\otimes_\Z \hat Z$ and is $\hat L$-translation invariant, i.e., $\phi(x)$ depends only on $x \mod \hat L$. Then
$$
S_L =\oplus_{\mu \in L'/L} \C \phi_\mu, \quad \phi_\mu =\cha (\mu +\hat L).
$$
It is easy to check that $S_L$ is $\SL_2(\Z)$-invariant under the Weil representation $\omega$, we denote this representation $\omega_L$.  One has by definition
\begin{align}
\omega_L(n(b)) \phi_\mu &= e(-b Q(\mu))\phi_\mu, \quad b \in \Z,
\\
\omega_L(w) \phi_\mu &= e(\frac{n-2}8) ([L':L])^{-\frac{1}2} \sum_{\nu \in L'/L} e( (\mu, \nu)) \phi_\nu.   \notag
\end{align}
Here
$$
n(b) =\kzxz {1} {b} {0} {1}, \quad w = \kzxz {0} {-1} {1} {0}
$$
and we have used the fact
$$
\psi_f(x) =\psi_\infty(-x) =e(-x)
$$
when $x \in \Q$. We also write
$$
\quad m(a) =\kzxz {a} {0} {0} {a^{-1}}.
$$
If we identify $S_L \cong \C[L'/L]=\oplus_{\mu \in  L'/L} \C e_\mu$ via $\phi_\mu \mapsto e_\mu$ and let $\rho_{L_-}$ be the Weil representation in \cite{Borcherds98} (also \cite{Bruinier02}) associated to the quadratic lattice $L_-$, where $L_-=L$ but with quadratic form $Q_-(x) =-Q(x)$, then  one sees immediately
\begin{equation}
\omega_L =\rho_{L_-}.
\end{equation}


 Recall that a meromorphic function  $f:\H\to S_L$ is called a {\em  weakly holomorphic modular form}
of weight $k$ with respect to $\SL_2(\Z)$ and $\omega_L$ if it
satisfies the following conditions.
\begin{enumerate}
\item[(i)]  One has $f \mid_{k,\omega_L} \gamma= f$
for all $\gamma=\abcd \in \Gamma$, where
$$
f \mid_{k,\omega_L} \gamma(\tau) = (c \tau + d)^{-k} \omega_L(\gamma)^{-1}f(\tau).
$$

\item[(ii)] There is a $S_L$-valued Fourier polynomial
\[
P_f(\tau)=\sum_{\mu\in L'/L}\sum_{n\leq 0} c(n,\mu)\, q^n\,\ph_\mu
\]
such that $f(\tau)-P_f(\tau)=O(e^{-\eps v})$ as $v\to \infty$ for
some $\eps>0$.
\end{enumerate}
The Fourier polynomial $P_f$  is called the {\em principal part}
of $f$. We denote the vector space of these weakly holomorphic modular
forms by  $M_{k,\omega_L}^!$. The Fourier expansion of any $f\in
M_{k,\omega_L}^!$ is of the form
\begin{equation}
\label{deff}
f(\tau)= \sum_{\mu\in L'/L}\sum_{\substack{n\in \Q\\ n\gg-\infty}} c(n,\mu)\, q^n\,\ph_\mu
\end{equation}
With this notation, we define
\begin{equation}
Z(f) = \sum_{n> 0, \mu\in L'/L} c(-n, \mu) Z(n, \mu).
\end{equation}
Here $Z(m, \mu)=Z(m, \phi_\mu)$.
 Let $S_L^\vee$ be
the dual space of $S_L$---the space of linear functionals on $S_L$,
and let $\{ \ph_\mu^\vee \} $ be the dual  basis in $S_L^\vee$ of
the basis $\{\ph_\mu \}$ of $S_L$.  Recall that   the Siegel theta
function (for $(z, h) \in X_K$)
$$
\theta_L(\tau, z, h) =\sum_{\mu} \theta(\tau, z,  h, \ph_\mu)\,
\ph_\mu^\vee
$$
is an $S_L^\vee$-valued holomorphic modular form of weight $0$ for
$\SL_2(\Z)$ and $\omega_L^{\vee} $ defined as follows (see \cite[Section 2]{BY09} or \cite{Kudla03} for details). For $z\in \mathbb D$, consider the orthogonal decomposition:
$$
V_\R = z \oplus z^\perp,  \quad x = x_{z} + x_{z^\perp}.
$$
Then for $\phi \in S(V_f)$ and $(z, h) \in X_K$, one defines
\begin{equation}
\theta(\tau, z, h, \phi) =v\sum_{ x \in V} \phi(h^{-1} x) e(\tau Q(x_{z^\perp}) + \bar\tau Q(x_z)).
\end{equation}
Here $v =\hbox{Im}(\tau)$ is the imaginary part of $\tau$.
Notice that $\theta(\tau, z, 1, \phi_\mu)= \overline{\theta(\tau, z, \mu)}$ in comparison with Borcherds'  Siegel theta functions.

 We consider the regularized theta integral
\begin{align}
\label{reg1} \Phi(z,h, f)=\int_{\calF}^{reg} \langle f(\tau),
\theta_L(\tau,z, h)\rangle\,d\mu(\tau) =\int_{\calF}^{reg}  \sum_{\mu \in L'/L}f_\mu(\tau) \theta(\tau, z, h, \phi_\mu)d\mu(\tau)
\end{align}
for $z\in \D$ and $h \in H(\A_f)$.  Here $\mathcal F$ is the
standard domain for $\SL_2(\mathbb Z) \backslash \mathbb H$, and   we write
$$
f(\tau) =\sum_{\mu \in  L'/L} f_\mu(\tau) \phi_\mu.
$$

The
integral is regularized as in \cite{Borcherds98}, that is,
$\Phi(z,h,f)$ is defined as the constant term in the Laurent
expansion at $s=0$ of the function
\begin{align}
\label{reg2} \lim_{T\to \infty}\int_{\calF_T} \langle f(\tau),
\theta_L(\tau,z,h)\rangle\,v^{-s} d\mu(\tau).
\end{align}
Here $\calF_T=\{\tau\in \H; \; \text{$|u|\leq 1/2$, $|\tau|\geq 1$,
and $v\leq T$}\}$ denotes the truncated fundamental domain and the
integrand \begin{equation} \langle f(\tau),\theta_L(\tau,z,h)\rangle
=\sum_{\mu \in L'/L} f_\mu(\tau) \theta(\tau, z, h, \ph_\mu)
\end{equation}
is the pairing of $f$ with the Siegel theta function, viewed as a
linear functional on the space $S_L$.
We remark that our regularized theta integral $\Phi(z, h, f)$ is exactly the same with the one in \cite{Borcherds98} and \cite{Bruinier02} when $h=1$.

The following is the first part of \cite[Theorem 13.3]{Borcherds98} (see also \cite[Theorem 3.22]{Bruinier02}) in our setting.

\begin{theorem} \label{BorcherdsLifting} Let $f(\tau) =\sum c(m, \mu) q^m \phi_\mu \in M_{1-\frac{n}2, \omega_L}^!$  be a weakly holomorphic modular form  of weight $1-\frac{n}2$ for $\SL_2(\Z)$ and $\omega_L$,  and assume that $c(m, \mu) \in \Z$ for $m <0$. Then there is a meromorphic modular form $\Psi(z, h, f))$  of  weight $k={c(0, 0)/2}$ on $X_K$  (with some characters) such that
\begin{enumerate}
\item  one has
$$
\mathrm{div} (\Psi(z, h,  f)^2 )= Z(f) =\sum_{m>0, \mu\in  L'/L} c(-m, \mu)  Z(m, \mu).
$$
Here we count $Z(m, \mu)$ with multiplicity $2$ or $1$ depending on whether  $2\mu \in  L$ or not.
\item  One has
$$-\log\| \Psi(z,h,  f)\|_{\Pet}^4= \Phi(z,h,  f).$$
Here $\| \,\|_{\Pet}$ is a suitably normalized Petterson norm.

\end{enumerate}
\end{theorem}


To describe the Borcherds product expansion formula for $\Psi(z, h, f)$, we need some  preparation. First, it works in each connected component. By  the strong approximation theorem, one has
$$
H(\A_f) =
  \coprod H(\Q)^+ h_j K,
$$
so
$$
X_K =\coprod X_{\Gamma_j} =\coprod \Gamma_j \backslash \mathbb D^+,
$$
where $\Gamma_j=H(\Q)^+ \cap h_j K h_j^{-1}$ and $\mathbb D^+$ is one of the two connected components of $\mathbb D$. In this decomposition, one has
$$
Z(m, \phi_\mu)= \sum_{j} Z_{L_j}(m, \mu_j),
$$
where $L_j= h_j L=h_j\hat L \cap V$, and $\mu_j \in L_j'/L_j$ with $\mu_j -h_j \mu \in  \hat L_j$, and
$$
Z_{L_j}(m, \mu_j) = \{ z \in \mathbb D^+:\,  (z, x) =0 \hbox{ for some } x \in \mu_j+L_j, Q(x) =m\}.
$$
In the following, we will stick with the irreducible component $X_\Gamma =\Gamma \backslash \mathbb D^+$ and the lattice $L$. The other components are the same.

Assume that $V$ has an isotropic line $\Q \ell$ (a cusp). We assume that $\ell\in L$ is primitive, i.e., $L\cap \Q\ell =\Z \ell$.  Choose $\ell' \in L'$  with $(\ell, \ell')=1$.  Assume further $(\ell, L) =N_\ell\Z$ and choose $\xi \in L$ with $(\ell, \xi) =N_\ell$.  Let $M= (\Q \ell + \Q \ell')^\perp \cap L$, and let
$$
L_0'=\{ x \in L' : \,  (\ell, x) \equiv 0 \mod (N_\ell) \} \supset L.
$$
Then there is a projection
\begin{equation}\label{projection}
p:  L_0' \rightarrow M', \quad p(x) = x_M + \frac{(x, \ell)}{N_\ell}  \xi_M,
\end{equation}
where $x_M$ and $\xi_M$ are the orthogonal projections of $x,\xi \in V$ to $M_\Q=M\otimes_\Z \Q$. The projection $p$ has the nice property $p(L) \subset M$ although it is not an orthogonal projection anymore (see \cite[Pages 40-41]{Bruinier02}). So it induces a projection from $L_0'/L$ to $M'/M$.

Next, we define the Weyl chamber for
$$f =\sum f_\mu \phi_\mu=\sum c(m, \mu) q^m \phi_\mu \in M_{1-\frac{n}2, \omega_L}^!.$$
 Define
\begin{equation}\label{restrict1}
f_M(\tau)=\sum_{\lambda \in M'/M} f_{\lambda}(\tau) \phi_{\lambda, M} =\sum c_M(m, \lambda)q^m \phi_{\lambda, M},
\end{equation}
where  $\phi_{\lambda, M}=\cha(\lambda+\hat M)$,
\begin{equation}\label{restrict2}
f_{\lambda}(\tau)=\sum_{\substack{\mu\in L'_0/L\\ p(\mu)=\lambda}}f_{\mu}(\tau).
\end{equation}
Then $f_M$ is a $S_M$-valued modular form by Borcherds \cite[Theorem 5.3]{Borcherds98} with Weil representation  $\omega_M$.


Let $\Gr(M)$ be  the set of negative lines in $M_\R$ (the Grassmannian), which is a real  manifold of dimension $n-1$ (as $M$ has signature $(n-1, 1)$). For $\lambda \in M'/M$ and $m \in \Q$ with $m \equiv Q(\lambda) \pmod 1$, let
$$
Z_M(m, \lambda) = \{ z \in \Gr(M):\, (z, x) =0 \hbox{ for some } x \in \lambda+M, Q(x) =m\},
$$
which is either empty or a real divisor of $\Gr(M)$. The {\it Weyl chamber} associated to a weakly holomorphic form $f \in  M_{1-\frac{n}2, \omega_{L}}^!$ is the connected components of  (see \cite[Page 88]{Bruinier02})
$$
\Gr(M) - \bigcup_{\mu \in L_0'/L} \bigcup_{\substack{m \in Q(\mu)+\Z\\ m >0,  c(-m, \mu) \ne 0}} Z_M(m, p(\mu)).
$$
Given a Weyl chamber $W$  associated to $f$, we define its {\it Weyl vector} $\rho(W, f)=\rho(W, f_M)\in M'$ following Borcherds as follows (\cite[Section 10.4]{Borcherds98}, see also \cite[Page 88]{Bruinier02}). Let $\bar W$ be the closure of $W$. If $M\cap \bar W$ is anisotropic, it was defined in \cite[Section 9]{Borcherds98} with correction and extension given recently in \cite[Section 5]{BS17}. We don't need it in this paper and refer to \cite{BS17} for details. When $M\cap\bar W$ is isotropic, choose an isotropic  $\ell_M \in  M \cap \bar W$ and $\ell_M'\in M'$ with $(\ell_M, \ell_M') =1$.
Let $P=M\cap (\Q\ell_M + \Q \ell_M')^\perp$, which is positive definite of rank $n-2$. Similar to the projection $p$ from $L_0'/L$ to $M'/M$, one has also a projection $p$ from $M_0'/M$ to $P'/P$ defined in the same way. For the same reason, we have the weakly holomorphic modular form $f_P$ (coming from $f_M$).
Define
\begin{align}
\rho_{\ell_M'}&=\text{constant term of}\ {\theta}_{P}(\tau)f_{P}(\tau)E_2(\tau)/24, \label{eq:WeylVector1}
\\
\rho_{\ell_M} &=-\rho_{\ell_M'}Q(\ell_M')-\frac{1}{4}\sum_{\begin{scriptsize}\begin{array}{c}\lambda\in M_0'/M\\p(\lambda)=0+P \end{array}\end{scriptsize}}
         c_M(0, \lambda)\textbf{B}_2((\lambda,\ell_M')) \\
  &\qquad -\frac{1}{2}\sum_{\begin{scriptsize}\begin{array}{c}\gamma\in P'\\(\gamma,W)>0 \end{array}\end{scriptsize}}\sum_{\begin{scriptsize}\begin{array}{c}\lambda\in M_0'/M\\ p(\lambda)= \gamma +P\end{array}\end{scriptsize}}
        c_M(-Q(\gamma), \lambda)\textbf{B}_2((\lambda,\ell_M'))), \notag
        \\
  \rho_P&=-\frac{1}{2}\sum_{\begin{scriptsize}\begin{array}{c}\gamma\in P'\cap M'\\(\gamma,W)>0 \end{array}\end{scriptsize}}c_M(-Q(\gamma), \gamma) \gamma,
  \\
  \rho(W, f) &= \rho_P + \rho_{\ell_M} \ell_M + \rho_{\ell_M'} \ell_M'. \label{eq:WeylVector2}
\end{align}
Here
$$
E_2=1-24\sum_{n>0}\sigma_{1}(n)q^n$$
is the weight $2$ Eisenstein series,  $\bold B_2(x)=\{x\}^2-\{x\}+\frac{1}{6}$ is the second Bernouli polynomial of $\{x\}$, where $ 0\le \{x \} = x -[x]<1$ is the fractional part of $x$.

Now we can state the beautiful product expansion formula of Borcherds as follows in the signature $(n, 2)$ case (\cite[Theorem 13.3]{Borcherds98}, see also \cite[Theorem 3.22]{Bruinier02}):

\begin{theorem} (Borcherds)  \label{BorcherdsProduct} Let the notation be as  above. Let $W$ be a Weyl chamber of $f$ whose closure contains $\ell_M$.  Then the memomorphic automorphic form $\Psi(z,  f)=\Psi(z, 1, f)$ has an infinite product expansion near the cusp $\Q\ell$ (more precisely, when $\hbox{Im}(z) \in W$ with $-Q(\hbox{Im}(z))$ sufficiently large).
$$
\Psi(z,  f)=C e((z, \rho(W, f))) \prod_{\substack{\lambda \in M' \\ (\lambda, W) >0}}
  \prod_{\substack{ \mu \in L_0'/L\\ p(\mu) \in \lambda+M}}\left[ 1 -e((\lambda, z)+(\mu, \ell'))\right]^{c(-Q(\lambda), \mu)}.
$$
Here $C$ is a constant with absolute value
\begin{equation} \label{eq:constant}
\left | \prod_{\delta \in \Z/N_\ell, \delta \ne 0}(1-e(\frac{\delta}N_\ell))^{\frac{c(0, \frac{\delta}N_\ell \ell)}2} \right |.
\end{equation}
\end{theorem}
\begin{proof} (sketch) We derive the formula from  \cite[Theorem 13.3]{Borcherds98}. Let $L_-=L$ with quadratic form $Q_-(x) =-Q(x)$ so that $L_-$ has signature $(2, n)$, for which we can apply Borcherds' theorem.  We use subscript $-$ to indicate the corresponding notation in Borcherds. First notice that  the symmetric domain $\mathbb D_-=\mathbb D$ and the tautological bundle $\mathcal L_-=\mathcal L$. Since $(\ell, \ell')=1$, one has $(-\ell, \ell')_-=1$. So the tube domains $\mathcal H_{\ell, \ell'}$ and $\mathcal H_{-\ell, \ell', -}$ are the same too. Furthermore, for $z \in \mathcal H_{\ell, \ell'}=\mathcal H_{-\ell, \ell', -}$, one has
$$
w_-(z) =z + \ell' -(Q_-(z) + Q_-(\ell')) (-\ell) = w(z).
$$
 Notice that $f|\kzxz {-1} {0} {0} {-1} =f$ implies
\begin{equation} \label{eq:sign}
c(m, \mu) =c(m,-\mu), \quad \hbox{  and  } \quad
c_M(m, \lambda) =c_M(m, -\lambda).
\end{equation}
Since $(\ell_M, \ell_M')=1$, one has $(\ell_M, -\ell_M')_-=1$. Using \cite[Theorem 10.4]{Borcherds98} (a minor mistake there missing the $\frac{1}4$ summation  part), one checks that $\rho_{\ell_M', -} =\rho_{\ell_M'}$ (as $\theta_{P, -}=\overline{\theta_P}$), and
\begin{align*}
\rho_{\ell_M, -} &= - \rho_{\ell_M',-} Q_-(\ell_M')
 +\frac{1}4 \sum_{\substack{\lambda \in M_0'/M, \\ p(\lambda) =0}} c_M(0, \lambda) \bold B_2( (\lambda, -\ell_M')_-)
 \\
  &\qquad + \frac{1}{2} \sum_{\substack {\gamma \in P' \\ (\gamma, W)_- >0}} \sum_{\substack{\lambda \in M_0'/M, \\ p(\lambda) =0+P}}
  c_M(Q_-(\gamma), \lambda) \bold B_2( (\lambda, -\ell_M')_-)
  \\
  &= \rho_{\ell_M'} Q(\ell_M') +\frac{1}4 \sum_{\substack{\lambda \in M_0'/M, \\ p(\lambda) =\gamma+P}} c_M(0, \lambda) \bold B_2((\lambda, \ell_M'))
  \\
   &\qquad + \frac{1}{2} \sum_{\substack {\gamma \in P' \\ (\gamma, W) >0}} \sum_{\substack{\lambda \in M_0'/M, \\ p(\lambda) =\gamma+P}}
  c_M(-Q(\gamma), \lambda) \bold B_2( (\lambda, \ell_M')) \\
 &=-\rho_{\ell_M}.
 \end{align*}
In the last identity, we substitute $\gamma$ by $-\gamma$ and $\lambda$ by $-\lambda$, and apply (\ref{eq:sign}).
Similarly, one checks $\rho_{P,-} =-\rho_P$.
So Borcherds' Weyl vector
$$
\rho(W, f)_-=\rho_{P, -} + \rho_{\ell_M, -}\ell_M + \rho_{\ell_M'}(- \ell_M') = -\rho(W, f).
$$
So \cite[Theorem 13.3]{Borcherds98} gives  for $z \in \mathcal H_{\ell, \ell'}$
\begin{align*}
\Psi(z, f)
 &=C e((z, \rho(W, f)_-)_-) \prod_{\substack {\lambda \in M' \\ (\lambda, W)_- >0}}
  \prod_{\substack {\mu \in  L_0'/L \\ p(\mu) =\lambda+M}} \left[1 - e( (\lambda, z)_- + (\mu, \ell')_-)\right]^{c(Q_-(\lambda), \mu)}
  \\
  &=C e( (z, \rho(W,f))) \prod_{\substack {\lambda \in M' \\ (\lambda, W) >0}}
  \prod_{\substack {\mu \in  L_0'/L \\ p(\mu) =\lambda+M}} \left[1 - e( (\lambda, z) + (\mu, \ell'))\right]^{c(-Q(\lambda), \mu)}
\end{align*}
as claimed. Here we again substitute $\lambda$ and $\mu$ by $-\lambda$ and $-\mu$, and apply (\ref{eq:sign}).

\end{proof}
\begin{remark}  It is worthwhile to make a few remarks to clear up some (potentially confusing) differences in different versions.
\begin{enumerate}
\item
The sign difference in the formula above and the formula in \cite[Theorem 13.3]{Borcherds98} (and \cite[Theorem 3.22]{Bruinier02}) is due to the fact that they use $L_-$ (signature $(2,n)$) while we use $L$.
 \item
 The condition $p(\mu) \in \lambda+M$  here  is a more explicit reinterpretation of Borcherds' condition $\mu|M=\lambda$ given by Bruinier (\cite[Theorem 3.22]{Bruinier02}).

 \item  The constant $C$ can be taken as the product in (\ref{eq:constant}) at a given cusp. However, once it is fixed, the constants at  other cusps are determined by this constant (they are in the same connected component).

 \item   The conditions that  $n \ge 3$ and  that $M$ is  isotropic in  \cite{Bruinier02} was for convenience  and not necessary
 .

 \item  The neighborhood near the cusp $\Q \ell$ where the product formula is valid can be made precise. We refer to \cite[Theorem 3.22]{Bruinier02} for details.

 \item At different cusps, the product formulae look different. This is similar to the different Fourier expansions of a modular form at different cusps.

\end{enumerate}
\end{remark}

\subsection{Big CM cycles, incoherent Eisenstein series, and the big CM value formula} \label{sect:CMvalue}

Let $E$ be  a CM number field of degree $n+2$ with the maximal totally real subfield $F$. Let $\sigma_i,  1\le i \le \frac{n}2+1$ be distinct real embeddings of $F$. Choose  an element $\alpha \in  F$ with $\sigma_{\frac{n}2+1}(\alpha) <0$ and $\sigma_i(\alpha) >0$ for all $1 \le i \le \frac{n}2$, and let
$W=E$ with the $F$-quadratic form $Q_F(z) =\alpha z \bar z$. Let $W_\Q=E$ with the $\Q$-quadratic form
$$
Q_\Q(z) =\tr_{F/\Q} Q_F(z) =\tr_{F/\Q}(\alpha z \bar z).
$$
Notice that $(W_\Q, Q_\Q)$ is a $\Q$-quadratic space of signature $(n, 2)$. Now we assume that $(W_\Q, Q_\Q) \cong  (V, Q)$, where $(V, Q)$ is a given $\Q$-quadratic space of signature $(n,2)$. Write $n_0=\frac{n}2 +1$. Then we have
$$
V_\R \cong \oplus_{1 \le i \le n_0} W_{\sigma_i},
$$
where $W_{\sigma_i} =W\otimes_{F,\sigma_i} \R$ has signature $(2, 0)$ or $(0, 2)$ according to $1\le i <n_0$ or $i=n_0$. The negative two plane $W_{\sigma_{n_0}}$ gives rise to two `big' CM points $z_{\sigma_{n_0}}^\pm$, which turns out to be defined over a finite extension of $\sigma_{n_0}(F)$. Define an algebraic torus $T$ over $\Q$ by the following diagram:
\begin{equation} \label{eq:torus}
\xymatrix{
 1 \ar[r]  &\G_m \ar[r] \ar[d] &T \ar[r] \ar[d] &\hbox{Res}_{F/\Q} \SO(W) \ar[r] \ar[d] &1
\cr
 1 \ar[r]  &\G_m \ar[r] &H \ar[r] &\SO(V) \ar[r] &1.
 \cr
}
\end{equation}
Then $T$ is a maximal torus in $H=\Gspin(V)$ (thus the  name  big CM points and big CM cycles). It is known (\cite[Section 2]{BKY12}) that
$$
Z(W, z_{\sigma_{n_0}}^\pm)
= \{ z_{\sigma_{n_0}}^\pm\} \times (T(\Q) \backslash T(\A_f)/K_T),  \quad K_T=K\cap T(\A_f)
$$
is a zero cycle in $X_K$ defined over $F$, called a big CM cycle of $X_K$. Let
$Z(W)$ be the formal sum of all its Galois conjugates (counting multiplicity), which  is a big CM cycle of $X_K$ over $\Q$. We refer to \cite[Section 2]{BKY12} for a more precise definition and basic properties of this cycle.

Associated to this quadratic space and the additive adelic character $\psi_F =\psi \circ \tr_{F/\Q}$ is a Weil representation $\omega=\omega_{\psi_F}$ of $\SL_2(\A_F)$ (and thus $T(\A_\Q))$ on $S(W(\A_F)) =S(V(\A_\Q))$. Let $\chi=\chi_{E/F} $ be the quadratic Hecke character of $F$ associated to $E/F$, then $\chi=\chi_W$ is also the quadratic Hecke character of $F$ associated to $W$, and there is a $\SL_2(\A_F)$-equivariant map
\begin{equation}
\lambda=\prod \lambda_v : S(W(\A_F)) \rightarrow  I(0, \chi), \quad  \lambda(\phi) (g) = \omega(g) \phi(0).
\end{equation}
Here $I(s, \chi) =\Ind_{B_{\A_F}}^{\SL_2(\A_F) }\chi |\cdot|^s$ is the principal series, whose sections (elements) are smooth functions $\Phi$  on $\SL_2(\A_F)$ satisfying the condition
$$
\Phi(n(b) m(a) g,s )= \chi(a)|a|^{s+1}\Phi(g, s), \quad b \in \A_F,  \hbox{ and  } a \in \A_F^\times.
$$
Here $B =NM$ is the standard Borel subgroup of $\SL_2$. Such a section is called factorizable if $\Phi=\otimes \Phi_v$ with $\Phi_v\in I(s, \chi_v)$. It is called standard if $\Phi|_{\SL_2(\hat{\OO}_F) \SO_2(\R)^{n_0}} $ is independent of $s$. For a standard  section $\Phi \in I(s, \chi)$, its associated Eisenstein series is defined as
$$
E(g, s, \Phi) = \sum_{\gamma \in B_F \backslash \SL_2(F)} \Phi(\gamma g, s)
$$
for $\Re(s) \gg 0$.

For $\phi \in S(V_f) =S(W_f)$, let $\Phi_f $ be the standard section associated to $\lambda_f(\phi) \in I(0, \chi_f)$. For each real embedding $\sigma_i: F \hookrightarrow \R$,  let  $\Phi_{\sigma_i} \in I(s, \chi_{\C/\R})=I(s, \chi_{E_{\sigma_i}/F_{\sigma_i}})$ be the unique `weight one' eigenvector of $\SL_2(\R)$ given by
$$
\Phi_{\sigma_i}(n(b)m(a) k_\theta) = \chi_{\C/\R}(a) |a|^{s+1} e^{i  \theta},
$$
for $b \in \R$, $a\in \R^\times$, and $k_\theta =\kzxz {\cos\theta} {\sin \theta} {-\sin \theta} {\cos \theta} \in \SO_2(\R)$. We define  for $\vec\tau =(\tau_1, \cdots, \tau_{n_0}) \in \H^{n_0}$
$$
E(\vec\tau, s, \phi) =  \norm(\vec v)^{-\frac12} E(g_{\vec\tau}, s, \Phi_f \otimes  (\otimes_{1 \le i \le n_0}\Phi_{\sigma_i} )),
$$
where $\vec v =\hbox{Im}(\vec\tau)$, $\norm(\vec v) =\prod_i v_i$, and $g_{\vec\tau} = (n(u_i) m(\sqrt{v_i}))_{1\le i \le n_0}$. It is a (non-holomorphic) Hilbert modular form of scalar weight $1$ for some congruence subgroup of $\SL_2(\OO_F)$. Following \cite{BKY12}, we further normalize
$$
E^*(\vec\tau, s, \phi) = \Lambda(s+1, \chi) E(\vec\tau, s, \phi),
$$
where $\partial_F$ is the different of $F$, $d_{E/F}$ is the relative discriminant of $E/F$, and
\begin{equation} \label{eq:L-series}
\Lambda(s, \chi) =A^{\frac{s}2} (\pi^{-\frac{s+1}2}
\Gamma(\frac{s+1}2))^{n_0} L(s, \chi),  \quad A=\norm_{F/\mathbb Q}
(\partial_F d_{E/F}).
\end{equation}
The Eisenstein series is incoherent in the sense that  $\Phi= \otimes \Phi_{v}$ is in the image of $\lambda$ on $S(\mathcal C)$, where $\mathcal C$ is an incoherent system of quadratic spaces over $F_v$, given  by $\mathcal C_v= W_v$ for all places $v$ except the one $v=\sigma_{n_0}$.   This  incoherence forces $E^*(\vec\tau, 0, \phi)=0$ automatically.


\begin{proposition} (\cite[Proposition 4.6]{BKY12}) \label{prop4.3} Let $\ph \in S(V_f) =S(W_f)$.
For  a totally positive element $t\in F^\times_+$,
let $a(t,\ph)$ be the $t$-th Fourier coefficient of $E^{*,
\prime}(\vec\tau, 0, \ph)$ and write
the constant term  of  $E^{*, \prime}(\vec\tau, 0, \ph)$ as
$$\ph(0)\Lambda(0, \chi)  \log \norm(\vec v) +a_0(\ph).$$
Let
$$
\mathcal E(\tau, \ph) =  a_0(\ph) + \sum_{n \in \mathbb
Q_{>0}} a_n(\ph)\, q^n
$$
where (for $n >0$)
$$
a_n(\ph) =\sum_{t \in F^\times_+, \, \tr_{F/\mathbb Q} t =n} a(t,\ph).
$$
Here $F_+^\times$ consists of all totally positive elements in $F$.
Then, writing $\tau^\Delta=(\tau, \cdots, \tau)$ for the diagonal image of $\tau\in \mathbb H$ in $\mathbb H^{n_0}$,
$$
E^{*, \prime}(\tau^\Delta, 0, \ph) -\mathcal E(\tau, \ph)-\ph(0)\, (\frac{n}2+1)\,\Lambda(0, \chi)\,
\log v
$$
is of exponential decay as  $v$ goes to infinity.
Moreover, for $n >0$
$$
a_n(\ph)  =\sum_p a_{n, p}(\ph) \log p
$$
with $a_{n, p}(\ph) \in \mathbb Q(\ph)$, the subfield of $\mathbb C$
generated by the values $\ph(x)$, $x\in V(\A_f)$.
\end{proposition}

\begin{remark} \label{rem:Constant}  There is a minor mistake in \cite[Proposition 4.6]{BKY12}) about the constant. The corrected form is
$$
E_0^{*, \prime}(\vec\tau, 0, \ph)=\ph(0)\Lambda(0, \chi)  \log \norm(\vec v) +a_0(\ph)
$$
(i.e., $a_0(\phi)$ might not be a multiple of $\phi(0)$). Direct calculation gives
$$
E_0^*(\vec\tau, s, \ph) = \ph(0) \Lambda(s+1, \chi) (\norm(\vec v))^{\frac{s}2} -(\norm(\vec v))^{-\frac{s}2} \Lambda(s, \chi) \tilde W_{0, f}(s, \ph)
$$
where     (when $\phi$ is factorizable)
$$
\tilde W_{0, f}(s, \ph)= \prod_{\mathfrak p \nmid \infty} \tilde W_{0, \mathfrak p}(s, \phi_\mathfrak p)
= \prod_{\mathfrak p \nmid \infty} \frac{|A|_\mathfrak p^{-\frac{1}2} L_\mathfrak p(s+1, \chi)}{\gamma(W_\mathfrak p) L_\mathfrak p(s, \chi)}  W_{0, \mathfrak p}(s, \phi_\mathfrak p)
$$
is the product of re-normalized local  Whittaker functions (see (\ref{eq:Whit})). With this notation, one has
\begin{equation} \label{eqa0}
a_0(\phi) = - \tilde{W}_{0, f}'(0, \phi) - 2\phi(0)  \Lambda'(0, \chi).
\end{equation}
\end{remark}

Notice that $a(t, \phi_\mu) =0$ automatically unless $\mu+\hat L$ represents $t$, i.e., $t -Q_F(\mu) \in  \partial_F^{-1}\OO_F$.
The following is a special case of the main CM value formula of Bruinier, Kudla, and Yang (\cite[Theorem 5.2]{BKY12}).

\begin{theorem} \label{theo:BigCM} Let
$$
f(\tau) =\sum_{\mu \in L'/L} f_\mu(\tau) \phi_\mu =\sum c(m, \mu) q^m \phi_\mu \in M_{1-\frac{n}2, \omega_L}^!
$$
with $c(0, 0)=0$, and let $\Psi(z, f)$ be its Borcherds  lifting. Then
$$
-\log| \Psi(Z(W), f) |^4 =C(W, K)\left(  \sum_{\substack {  \mu \in  L'/L,\\ m\ge 0 \\ m \equiv Q(\mu) \pmod 1 }} c(-m, \mu) a_m(\phi_\mu)    \right).
$$
 Here
$$
C(W, K)=\frac{ \deg(Z(W,z_{\sigma_2}^\pm))}{  \Lambda(0, \chi)}.
$$
\end{theorem}

To compute the $t$-th Fourier coefficient $a(t, \phi)$ of $E^{*, \prime}(\vec \tau, 0, \phi)$, one may assume that $\phi =\otimes \phi_{\mathfrak p}$ is factorizable by linearity.  Write for $t \ne 0$
$$
E_t^*(\vec\tau, s, \phi) =\prod_{\mathfrak p \nmid \infty} W_{t, \mathfrak p}^*(s, \phi) \prod_{j=1}^{n_0} W_{t, \sigma_j}^*(\tau_j, s, \Phi_{\sigma_j}),
$$
where
$$
W_{t, \mathfrak p}^*(s, \phi)= |A|_\mathfrak p^{-\frac{s+1}{2}}  L_\mathfrak p(s+1, \chi_\mathfrak p) W_{t, \mathfrak p}(s, \phi)
$$
for finite prime $\mathfrak p$ with
\begin{equation} \label{eq:Whit}
W_{t, \mathfrak p}(s, \phi)= \int_{F_\mathfrak p} \omega(wn(b))(\phi_\mathfrak p) (0) |a(wn(b))|_\mathfrak p^s\psi_\mathfrak p(-t b) \, db,
\end{equation}
and for infinite prime $\sigma_j$
$$
W_{t, \sigma_j}^*(\tau_j, s, \Phi_{\sigma_j})= v_j^{-1/2} \pi^{-\frac{s+2}2} \Gamma(\frac{s+2}2) \int_{\R} \Phi_{\sigma_j}(w n(b)g_{\tau_j}, s) \psi(-bt) db.
$$
Here $A$ is defined in (\ref{eq:L-series}) and $|a(g)|_\mathfrak p=|a|_\mathfrak p $ if $ g =n(b) m(a) k$ with $k\in \SL_2(\OO_\mathfrak p)$.

The following proposition is well-known and is recorded here for reference. Recall $W=E$ with $Q_F(z) =\alpha z \bar z$, $\alpha \in F^\times$.
\begin{proposition} \label{prop:Whittaker} For a totally positive number $t \in F^+$, let
$$
\hbox{Diff}(W, t) =\{ \mathfrak p :\,  W_{\mathfrak p} \hbox{ does not represent } t\}
$$
be the so-called `Diff' set of Kudla. Then $|\hbox{Diff}(W, t)|$ is finite and odd. Moreover,
\begin{enumerate}
\item If $|\hbox{Diff}(W, t)|>1$, then $a(t,\phi)=0$.

\item If  $\hbox{Diff}(W, t)=\{\mathfrak p\}$,  then $W_{t, \mathfrak p}^*(0,\phi)=0$, and
$$
a(t, \phi) =(-2i)^{n_0}  W_{t, \mathfrak p}^{*, \prime} (0,\phi) \prod_{\mathfrak q \nmid \mathfrak p\infty} W_{t, \mathfrak q}^*(0, \phi).
$$

\item When $\mathfrak p \nmid \alpha A$ is unramified in $E/F$,  and $\phi_\mathfrak p=\cha(\OO_{E_\mathfrak p})$, $W_{t, \mathfrak p}^*(s, \phi)=0$ unless $t \in \partial_F^{-1} $. In this case, one has
$$
\frac{W_{t, \mathfrak p}^*(0, \phi)}{\gamma(W_\mathfrak p)} = \begin{cases}
  1+ \ord_\mathfrak p(t\sqrt{D} ) &\ff \mathfrak p \hbox{ is split in } E,
  \\
  \frac{1+(-1)^{\ord_{\mathfrak p}(t\sqrt{D} )}}2 &\ff  \mathfrak p \hbox{ is inert in } E.
  \end{cases}
$$
Here $\gamma(W_{\mathfrak p})$ is the local Weil index (a $8$-th root of unity) associated to the Weil representation.
Moreover,  in this case,  $W_{t, \mathfrak p }^*(0, \phi)=0$ if and only if  $\ord_{\mathfrak p} (t\sqrt{D}) $ is odd and $\mathfrak p$ is inert in $E$.  In such a case, one has
$$
\frac{W_{t, \mathfrak p}^{*, \prime}(0, \phi)}{{\gamma(W_\mathfrak p)}} = \frac{1+\ord_\mathfrak p (t\sqrt{D}) }2 \log \norm(\mathfrak p).
$$

\item  One has for $1 \le j \le n_0$
$$
W_{t, \sigma_j}^*(\tau, 0, \Phi_{\sigma_j})=-2i e(t\tau),   \quad t >0
$$
and
$$
W_{ 0, \sigma_j}^*(\tau, s, \Phi_{\sigma_j}) = -i \pi^{-\frac{s+1}2} \Gamma(\frac{s+1}2) v^{-\frac{s}2}.
$$
\end{enumerate}
\end{proposition}
\begin{proof} (sketch) The Diff set is first defined by Kudla in \cite{KuAnnals}. In our case, the incoherent collection of $F_v$-quadratic spaces is $\{\mathcal C_v\}$ where $\mathcal C_v =W_v$ for $v \ne \sigma_{n_0}$ and $\mathcal C_{n_0}$ positive definite.  The archimedian places are not in the Diff set as $t$ is totally positive.  Let $\psi_F'(x)=\psi_F(\frac{x}{\sqrt D})$ and $W'=W$ with $F$-quadratic form $Q_F'(x) =\sqrt D Q_F(x) =x \bar x$. Then
 one has as Weil representations on  $S(W_f) =S(W_f')$:
 $$
 \omega_{W, \psi_F} =\omega_{W', \psi_F'},
 $$
 and thus the Whittaker functions have the following relation
 $$
W_{t, \mathfrak p}^{\psi_F}(s, \phi) = |\sqrt D|_{\mathfrak p}^{\frac{1}2} W_{t\sqrt{D}, \mathfrak p}^{\psi'}(s, \phi)
$$
for each prime $\mathfrak p$ of $F$.
 Recall that $W_{t, \mathfrak p}^*(0, \phi)=0$ if $\mathfrak p \in  \Diff(W, t)$. So (1) is obvious.
 Claim (3) follows from \cite[Proposition 2.1]{Yang05}. Claim  (4)  is a special case of \cite[Proposition 2.6]{KRYtiny}. Claim (2) follows from
 $$
 E_t^*(\vec \tau, s, \phi)= \prod_{\mathfrak p \nmid \infty} W_{t, \mathfrak p}^*(s, \phi)  \prod_{j=1}^{n_0}W_{t, \sigma_j}^*(\tau_j, s, \Phi_{\sigma_j})
 $$
and (4).

\end{proof}

\section{Product of modular curves and its diagonal divisor} \label{sect:BigCM} \label{sect:Product}

\subsection{Product of modular curves as a Shimura variety of orthogonal type $(2, 2)$} \label{sect2.1}
Let $N$ be a positive integer, and
Let $V=M_2(\Q)$ with the quadratic form  $Q(X) =N\det X$.
 Let $H$ be the algebraic group over $\Q$
$$
H =\{ (g_1, g_2) \in \GL_2 \times \GL_2 :\, \det g_1 =\det g_2\}.
$$
Then $H\cong \Gspin(V)$ and acts on $V$ via
$$
(g_1, g_2) X = g_1 X g_2^{-1}.
$$
One has the exact sequence
$$
1 \rightarrow  \mathbb G_m \rightarrow H \rightarrow \SO(V) \rightarrow 1.
$$
Recall the Hermitian symmetric domain $\mathbb D$ and the tautological line bundle $\mathcal L$  in Section  \ref{sect:Borcherds}. For a tube domain,
take an isotropic matrix $\ell= \kzxz{0}{-1}{0}{0} \in L$ and $\ell' =\kzxz {0} {0} {\frac{1}N} {0} \in V$ with $(\ell, \ell')=1$. Then the associated tube domain is
$$
\mathcal H_{\ell, \ell'} =\{ \kzxz {z_1} {0} {0} {-z_2} : \,  y_1 y_2  >0\} ,\quad y_i=\hbox{Im}(z_i),
$$
together with
$$
w: \mathcal H_{\ell, \ell'} \rightarrow \mathcal L,  \quad w( \kzxz {z_1} {0} {0} {-z_2}) = \kzxz {z_1} {-N z_1 z_2} { \frac{1}{N}} { -z_2}.
$$
Now the following proposition is clear.
\begin{proposition} \label{prop3.1} Define
$$
w_N: \,  \H^2 \cup (\H^-)^2  \rightarrow \mathcal L, \quad w_N(z_1, z_2) = \kzxz {z_1} {-z_1 z_2} {1} {-z_2} = Nw( \kzxz {\frac{z_1}{N}} {0} {0} {-\frac{z_2}N}) ,
$$
and
$$
pr:  \mathcal L \rightarrow \mathbb D,  \quad x + iy \mapsto z=\R x + \R (-y).
$$
Then $pr$ gives an isomorphism between $\mathcal L/\C^\times$ and $\mathbb D$, and the composition $pr \circ w$ gives an isomorphism between $\H^2 \cup (\H^-)^2$ and $\mathbb D$. Moreover, $w_N$ is $H$-equivariant, where $H\subset \GL_2(\R) \times \GL_2(\R)$ acts on $\H^2 \cup (\H^-)^2$ via  the usual linear fraction transformations:
$$
(g_1, g_2)(z_1, z_2) =(g_1(z_1), g_2(z_2)),
$$
and acts on $\mathcal L$ and $\mathbb D$ naturally via its action on $V$. Moreover, one has

\begin{equation} \label{eq:linebundle}
(g_1, g_2) w_N(z_1, z_2) = \frac{(c_1 z_1+d_1)(c_2 z_2 +d_2)}{\nu(g_1, g_2)} w_N(g_1(z_1), g_2(z_2)),
\end{equation}
where $\nu(g_1, g_2) =\det g_1 =\det g_2$ is the spin character of $H=\Gspin(V)$. So
$$
j(g_1,g_2, z_1, z_2) = (c_1 z_1+d_1)(c_2 z_2 +d_2).
$$
\end{proposition}

For a congruence subgroup $\Gamma$ of $\SL_2(\Z)$, let $X_\Gamma$ be the associated open modular curve over $\Q$ such that $X_\Gamma(\C) =\Gamma \backslash \H$.
Assume $\Gamma \supset \Gamma(M)$ for some integer $M \ge 1$.
Let
$$
\nu:  \A^\times \hookrightarrow \GL_2(\A), \quad \nu(d) = \diag(1, d).
$$
Let
$K(\Gamma)$ be the product of $\nu(\hat{\Z}^\times)$ and the preimage of $\Gamma/\Gamma(M)$ in $\GL_2(\hat\Z)$ (under the map $\GL_2(\hat\Z)\rightarrow \GL_2(\Z/M))$. Let $K=(K(\Gamma)\times K(\Gamma))\cap H(\A_f)$. Then    one has by the strong approximation theorem
$$
X_K\cong X_\Gamma \times X_\Gamma.
$$
In this way, we have identified the product of two copies of a modular curve $X_\Gamma$  with a Shimura variety $X_K$. We will fix this $K$ for a given congruence subgroup $\Gamma$ in this paper.
 The tautological line bundle $\mathcal L$ descends to a line bundle $\mathcal L_K=K\backslash \mathcal L$ of modular forms of 2 variables of weight $(1, 1)$ by (\ref{eq:linebundle}).

Let $L$ be an even integral lattice of $V$, and let $L'$ be its dual with respect to the quadratic form $Q$. We assume that $\Gamma \times \Gamma$ acts on $L'/L$ trivially. Then for $\mu \in L'/L$ and a rational number $m>0$ (and $m \equiv Q(\mu) \pmod 1$), the associated special divisor $Z(m, \mu)=Z_K(m, \mu)$ is given in this special case by
\begin{equation}\label{divisor}
Z(m, \mu) = (\Gamma \times \Gamma) \backslash \{ (z_1, z_2) \in  \H^2:\, w_N(z_1, z_2) \perp x \hbox{ for some } x \in  \mu+L, Q(x)=m\}.
\end{equation}
Alternatively,  $Z(m, \mu)$ is the sum of $Z(x)$, where $x \in  \mu+L$ with $Q(x)=m$ modulo  the action of $\Gamma \times \Gamma$. Here $Z(x)$ is the subvariety of $X_K$ given by $x^\perp$ (of signature $(1, 2)$):
$$
Z(x) = ( \Gamma \cap x^{-1} \Gamma x)\backslash \H  \cong  (\Gamma \times \Gamma)_x \backslash \{ (x z, z): \, z \in \H\}, \quad  [z] \mapsto [xz, z].
$$
The linear combinations of these divisors $Z(m, \mu)$ are called the special divisors of $X_K$.

\begin{lemma} \label{lem:diagonal} Let $\Gamma =\Gamma(N)$ and $L=M_2(\Z)$ with $Q(X) =N \det X$. For each $\gamma \in \SL_2(\Z)$, let
$$
Z_N(\gamma) = \{(\gamma z,  z) \in X_{\Gamma(N)} \times X_{\Gamma(N)}:\,  z \in X_{\Gamma(N)}\} \subset X_{\Gamma(N)} \times X_{\Gamma(N)}.
$$
Then $Z_N(\gamma)=Z(\frac{1}{N}, \frac{1}N\gamma +L)$ is a special divisor of $X_K$.
\end{lemma}
We denote by  $X(N)$  the compactification of $X_{\Gamma(N)}$ (to be compatible with usual definition of $X(N)$)
\begin{proof} If $x \in \frac{1}N\gamma +L$ with $Q(x) =N \det x =1/N$, then $Nx \in \gamma +N L$ and $\det (N x)=1$. So $Nx \in \SL_2(\Z)$, and
$Nx \gamma^{-1}=\gamma_1 \in \Gamma(N)$, and $ x = \gamma_1 (\frac{1}N \gamma)$. This implies
$$Z(\frac{1}{N}, \frac{1}N\gamma +L)=Z_N(\gamma).$$

\end{proof}

\begin{corollary} \label{cor:diagonal} Let $X_\Gamma^{\Delta}$ be the diagonal embedding of $X_\Gamma$ into $X_\Gamma \times X_\Gamma$. The $X_\Gamma^\Delta$ is a special divisor of $X_\Gamma \times X_\Gamma$  in the following sense.
 Assume   $\Gamma \supset \Gamma(N)$, we take $L=M_2(\Z)$ with $Q(X) =N \det$. Then the preimage of $X_\Gamma^\Delta$ in $X_{\Gamma(N)} \times X_{\Gamma(N)}$ is equal to
$$
\sum_{\gamma \in \Gamma/\Gamma(N) } Z_N(\gamma)
$$
in the notation of Lemma \ref{lem:diagonal}.
\end{corollary}

\subsection{Products of CM cycles as  big CM cycles}
For $j=1,2$, let $\kk_j =\Q(\sqrt d_j)$ with ring of integers $\OO_j=\Z[\frac{d_j +\sqrt{d_j}}2]$ of discriminant $d_j<0$  with $(d_1, d_2)=1$. In this subsection, we describe how to view a pair of CM points $(\tau_1, \tau_2) \in X_\Gamma \times X_\Gamma$ associated to $\kk_1$ and $\kk_2$  as a big CM  point in $X_K$ in the sense of \cite{BKY12}.
For this purpose,  let $\kk =\kk_1 \otimes_\Q  \kk_2=\Q(\sqrt{d_1}, \sqrt{d_2})$ with ring of integers $\OO_E =\OO_1 \otimes_\Z \OO_2$. Then $\kk$ is a biquadratic  CM  number field with  real quadratic subfield $F=\Q (\sqrt D)$ and $D=d_1d_2$.

 We define $W= \kk$ with the $F$-quadratic form $Q_F(x) =\frac{ Nx \bar x}{\sqrt D}$. Let $W_\Q =W$ with the $\Q$-quadratic form $Q_\Q(x) =\tr_{F/\Q} Q_F(x)$.  Let $\sigma_1=1$ and $\sigma_2 =\sigma$ be two real embeddings of $F$ with $\sigma_j(\sqrt D) =(-1)^{j-1}\sqrt D$. Then
 $W$ has signature $(0, 2)$ at $\sigma_2$ and $(2, 0)$  at $\sigma_1$ respectively, and so $W_\Q$ has signature $(2,2)$.
 Choose a $\Z$-basis of $\OO_E$ as follows
 $$
 e_1=1\otimes 1, \quad e_2 = \frac{-d_1+\sqrt{d_1}}2 =\frac{-d_1+\sqrt{d_1}}2 \otimes 1, \quad  e_3= \frac{d_2+\sqrt{d_2}}2 =1\otimes \frac{d_2+\sqrt{d_2}}2, \quad e_4= e_2 e_3.
 $$
We will drop $\otimes$ when there is no confusion.  Then it is easy to check that
\begin{equation} \label{eq:SpaceIdentification}
(W_\Q, Q_\Q) \cong (V, Q)=(M_2(\Q),  N\det),\quad\quad\quad \sum x_i e_i \mapsto \kzxz {x_3} {x_1} {x_4} {x_2}.
\end{equation}
We will identify $(W_\Q, Q_\Q)$ with the quadratic space $(V, Q) =(M_2(\Q), N\det)$. Under this identification, the lattice $L=M_2(\Z)$ becomes $\OO_E$.  Then one can check  that the maximal torus $T$ in (\ref{eq:torus}) can be identified with (\cite{HYbook}, \cite[Section 6]{BKY12})
$$
T(R) =\{ (t_1, t_2)\in (E_1\otimes_\Q R)^\times \times (E_2\otimes_\Q R)^\times:\,  t_1 \bar{t}_1 = t_2 \bar{t}_2 \},
$$
for any $\Q$-algebra $R$,  and the map from $T$ to $\SO(W)$ is given by  $(t_1, t_2) \mapsto {t_1}/{\bar{t}_2}$. The map from $T$ to $H$ is explicitly given as follows.
Define the embedding
\begin{equation} \label{eq:iota}
\iota_j: \kk_j \rightarrow  M_2(\Q), \quad  \iota_j(r) (e_{j+1}, e_1)^t = ( r e_{j+1}, r e_1)^t.
\end{equation}
Then $\iota= (\iota_1, \iota_2)$ gives the embedding from $T$ to $H$.

Extend the two real embeddings of $F$ into a CM type   $\Sigma=\{\sigma_1, \sigma_2\}$ of $E$ via
$$
\sigma_1(\sqrt{d_i}) =\sqrt{d_i} \in \H,\quad  \sigma_2(\sqrt{d_1}) =\sqrt{d_1}, \quad  \sigma_2(\sqrt{d_2}) =-\sqrt{d_2}.
$$
 Since $W_{\sigma_2} =W\otimes_{F, \sigma_2} \R \subset  V_\R$ has signature $(0, 2)$, it gives two points $z_{\sigma_2}^\pm$ in $\D$. In this case, the big CM cycles in Section \ref{sect:CMvalue} become
\begin{equation}
Z(W, z_{\sigma_2}^\pm) = \{ z_{\sigma_2}^\pm\} \times T(\Q)\backslash T(\A_f)/K_T  \in Z^2(X_K),
\end{equation}
and
$$
Z(W) = Z(W, z_{\sigma_2}^\pm) +  \sigma(Z(W, z_{\sigma_2}^\pm)).
$$
For simplicity, we will denote  $z_{\sigma_2}$ for $z^+_{\sigma_2}$.

\begin{lemma} \label{lem3.4}  On $\H^{\pm, 2}$, one has  $z_{\sigma_2} = (\tau_1, \tau_2) \in \H^2$ and $z_{\sigma_2}^-=(\bar \tau_1, \bar \tau_2)\in (\H^- )^2$, where
$$
\tau_j=\frac{d_j + \sqrt{d_j}}2.
$$
\end{lemma}
\begin{proof} In the decomposition
$$
V_\R =V\otimes_\Q \R =W_{\sigma_1} \oplus W_{\sigma_2}, \quad  W_{\sigma_i} =E\otimes_{F, \sigma_i}\R \cong \C,\ r\mapsto \sigma_i(r),
$$
the $\R$-basis $\{e_i, i =1, 2, 3, 4\}$  becomes
$$
e_1=(1, 1), \quad e_2=(-\bar \tau_1,-\bar \tau_1), \quad \quad e_3=(\tau_2, \bar\tau_2), \quad  \hbox{and}  \quad e_4=(-\bar\tau_1\tau_2, -\bar \tau_1 \bar\tau_2).
$$
The negative two plane $W_{\sigma_2}$ representing $z_{\sigma_2}^{\pm}$ has an $\R$-orthogonal basis
$$
u=(0, \sqrt{|d_2|}) ,  \hbox{ and }  v=(0, \sqrt{d_2}) \in W_{\sigma_2} \subset V_\R .
$$
One checks
\begin{align*}
u&= -\frac{D-\sqrt D}{2\sqrt{|d_1|}} e_1- \frac{d_2}{\sqrt{|d_1|}} e_2 +\frac{d_1}{\sqrt{|d_1|}} e_3  +\frac{2}{\sqrt{|d_1|}} e_4
\\
&= \kzxz { \frac{d_1}{\sqrt{|d_1|}}} {-\frac{D-\sqrt D}{2\sqrt{|d_1|}} }  {\frac{2}{\sqrt{|d_1|}}} {-\frac{d_2}{\sqrt{|d_1|} }}  .
\end{align*}
and
\begin{align*}
v&= \frac{\sqrt{d_2} (\sqrt{d_1} +\sqrt{d_2})}{2} e_1+  \frac{\sqrt{d_2}}{\sqrt{d_1}} e_2 - e_3
\\
 &=\kzxz {-1} { \frac{\sqrt{d_2} (\sqrt{d_1} +\sqrt{d_2})}{2} } {0}  {\frac{\sqrt{d_2}}{\sqrt{d_1}}  }   .
\end{align*}
So
$$
u-i v=\frac{2}{\sqrt{|d_1|} } \kzxz {\tau_1 }  {-\tau_1 \tau_2} {1} {-\tau_2}=\frac{2}{\sqrt{|d_1|}} w_N(\tau_1, \tau_2),
$$
and
$$
u+i v=\frac{2}{\sqrt{|d_1|}} w_N(\bar{\tau}_1, \bar{\tau}_2)
$$
as claimed.
\end{proof}

\begin{lemma} \label{lem:classprojection} Let
$K_j=\iota_j^{-1}(K(\Gamma))$ and let $\Cl(K_j) = E_j^\times \backslash E_{j, f}^\times/K_j$ be the associated class group of $E_j$. Then there is an injection
$$
p': T(\Q) \backslash T(\A_f)/K_T \rightarrow  \Cl(K_1)\times \Cl(K_2)
$$
with image
\begin{align*}
\IM(p') &=\{ (C_1, C_2) \in \Cl(K_1) \times \Cl(K_2):\, \exists\   t_j \in E_{j, f}^\times \hbox{ with } C_j=[t_j], t_1 \bar t_1=t_2 \bar t_2\}
\\
 &= \{ (C_1, C_2) \in \Cl(K_1) \times \Cl(K_2):\,  \exists  \hbox{  fractional ideals  } \mathfrak a_i \hbox{ with } C_j=[\mathfrak a_j],  \norm(\mathfrak a_1)=\norm(\mathfrak a_2)\}.
\end{align*}
\end{lemma}
\begin{proof} Clearly $p'$ is a group homomorphism.  We first check that $p'$ is injective. Assume $[t_1, t_2]\in \ker p$,  write $t_j=g_j k_j$ with $g_j \in E_j^\times$ and $k_j\in K_j$. Then $t_1 \bar t_1 =t_2 \bar t_2$ implies
$$
  \frac{g_1 \bar g_1}{g_2 \bar g_2} =\frac{ k_2 \bar k_2 }{ k_1 \bar k_1 } \in \Q_{>0}\cap \hat\Z^\times =\{1\},
  $$
 so $(g_1, g_2) \in T(\Q)$, and $ k_2 \bar k_2 =k_1 \bar k_1$. This implies
 $$(k_1, k_2) \in K_T=\{ (t_1, t_2) \in T(\A_f): (\iota_1(t_1), \iota_2(t_2))\in K_\Gamma=  K(\Gamma) \times  K(\Gamma)\}.
 $$
 So $(t_1, t_2) \in T(\Q) K_T$. The first formula for $\IM(p)$ is the definition. To show the second formula, assume $\norm(\mathfrak a_1) =\norm(\mathfrak a_2)$. Let  $t_j \in E_{j, f}$ such that its associated ideal is $\mathfrak a_j$. Then  $t_1 \bar t_1= t_2 \bar t_2 u$ for some  $u \in \hat\Z^\times$. When $p \nmid d_j$, $u_p =w_p \bar w_p$ for some $w_p\in \OO_{E_{j, p}}^\times$. So we can decompose $u=u_1^{-1}u_2$ such that $u_j=w_j \bar w_j \in \norm_{E_j/\Q} \hat{\OO}_{E_j}^\times$. Replacing $t_j$ by $t_j w_j$, we find $t_j\in E_{j, f}^\times$ such that $t_1\bar t_1 =t_2 \bar t_2$ and $[t_j]=[\mathfrak a_j]$.
\end{proof}

  Let $H_j$ be the  class field of $\kk_j$ associated to $K_j$ and $H=H_1H_2$, the composition of  $H_1$ and $H_2$.
By the complex multiplication theory, the point $[z_{\sigma_2}]\in X_K$ is defined over $H$. Moreover, one has a natural map induced by $\iota_j$ in (\ref{eq:iota})
\begin{equation} \label{eq:CMpoint}
\iota_j:  \Cl(K_{j}) \rightarrow X_\Gamma=\GL_2(\Q) \backslash \H^\pm \times \GL_2(\A_f)/K(\Gamma),  \quad \iota_j([t^{-1}])  =[\tau_j, \iota_j(t^{-1})]= \tau_j^{\sigma_{t}}
\end{equation}
Here $\sigma_{t} \in \Gal(H_j/E_j)$ is associated to $[t]$ by class field theory.  The last identity is Shimura's reciprocity law (see for example \cite{YaGalois}). We will also write
$\tau_j^{\sigma_t}=\tau_j^{\sigma_{\mathfrak a}}$ in ideal language where $[\mathfrak a_j]\in \Cl(K_j)$ corresponds to the idele  class of $t$. Now  the following two propositions are  clear.

\begin{proposition} \label{prop3.2}  Let $(t_1, t_2) \in T(\A_f)$, and let $\sigma_{t_j} \in \Gal(H_j/E_j)$ be the associated Galois element (to $t_j$) via the Artin map. Then
$$
[z_{\sigma_2}, (t_1^{-1}, t_2^{-1}) ] = [\tau_1^{\sigma_{t_1}}, \tau_2^{\sigma_{t_2}}].
$$
\end{proposition}

\begin{proposition} \label{prop:BigCMPoint} \label{prop3.7} Assume $(d_1, d_2) =1$.
Then
\begin{align*}
Z(W, z_{\sigma_2}) &=\sum_{ ([\mathfrak a_1], [\mathfrak a_2])\in \IM(p') } [\tau_1^{\sigma_{\mathfrak a_1}}, \tau_2^{\sigma_{\mathfrak a_2}}],
\\
 Z(W, z_{\sigma_2}^-) &=\sum_{ ([\mathfrak a_1], [\mathfrak a_2])\in \IM(p') } [(-\bar{\tau}_1)^{\sigma_{\mathfrak a_1}}, (-\bar{\tau_2})^{\sigma_{\mathfrak a_2}}],
 \\
Z(W)&=\sum_{ ([\mathfrak a_1], [\mathfrak a_2])\in \IM(p') }
 \left([\tau_1^{\sigma_{\mathfrak a_1}}, \tau_2^{\sigma_{\mathfrak a_2}}] +[(-\bar{\tau}_1)^{\sigma_{\mathfrak a_1}}, \tau_2^{\sigma_{\mathfrak a_2}}]
 + [\tau_1^{\sigma_{\mathfrak a_1}}, (-\bar{\tau}_2)^{\sigma_{\mathfrak a_2}}] +  [(-\bar{\tau}_1)^{\sigma_{\mathfrak a_1}}, (-\bar{\tau_2})^{\sigma_{\mathfrak a_2}}]\right).
\end{align*}
\end{proposition}

The following lemma will be used later.

\begin{lemma} \label{lem3.3} Assume again $(d_1, d_2)=1$. Let $\Cl(E_j)$ be the ideal class group of $E_j$. Let $C_j \in  \Cl(E_j)$ be an ideal class for each $j=1, 2$. Then there is an ideal $\mathfrak a_j \in C_j$ such that $\norm (\mathfrak a_1)=\norm(\mathfrak a_2)$. In  particular, when $K_j=\hat{\OO}_{E_j}$ in Lemma \ref{lem:classprojection}, then the map $p'$ is an isomorphism.
\end{lemma}
\begin{proof} We first show $H_1\cap H_2=\Q$. Let $p$ be a rational prime, then  $p\nmid d_1$ or $p\nmid d_2$. When $p\nmid d_j$, $p$ is unramified in $H_j$ and thus in $H_1\cap H_2$. So every prime $p$ is unramified in $H_1\cap H_2$, and thus $H_1\cap H_2=\Q$. This implies
$$
\Gal(H/\Q) \cong \Gal(H_1/\Q) \times \Gal(H_2/\Q).
$$
So there is $\sigma \in \Gal(H/\Q)$ such that $\sigma|H_j=\sigma_{C_j}$. In  particular, $\sigma \in  \Gal(H/E)$, which is abelian. By  the class field theory, there is an ideal $\mathfrak a$ of $E$ such that $\sigma_\mathfrak{a} =\sigma$.
Let $\mathfrak a_j =\norm_{E/E_j}\mathfrak a$. Then $\sigma|H_j=\sigma_{\mathfrak a_j}$ and $\norm(\mathfrak a_1) =\norm(\mathfrak a_2) =\norm(\mathfrak a)$. Moreover, one has $C_j=[\mathfrak a_j]$.
\end{proof}

\section{Gross and Zagier's singular moduli factorization formula} \label{GrossZagier}

   We will give a different  proof of Gross and Zagier's factorization formula (Theorem \ref{theo:GrossZagier}) in this section. For this, we take $L=M_2(\Z)$ with $Q(X) =\det X$, and $W=E$ with $Q_F(x) =\frac{x \bar x}{\sqrt D}$, where $E=\Q(\sqrt{d_1}, \sqrt{d_2})$ and $F=\Q(\sqrt D)$ are as in Section \ref{sect:Product}. In this case, the lattice  $L\cong \OO_E$ is unimodular.

{\bf Proof of Theorem \ref{theo:GrossZagier}}
Recall the identification at the  beginning of Section \ref{sect:Product} of the product $X_0(1) \times X_0(1)$  of modular curves with the orthogonal Shimura surface of signature $(2, 2)$ and the  isotropic vectors
 $\ell=\kzxz{0}{-1}{0}{0}$ and $\ell'=\kzxz {0} {0} {1} {0}$ used for the identification. We also use them as in Theorem  \ref{BorcherdsProduct}  for Borcherds product expansion.  Write
$$
j(\tau)-744 =\sum_{m \ge -1} c(m) q^{m},
$$
then Borcherds proved in \cite{Borcherds95}
$$
j(z_1) -j(z_2) = \Psi(j(\tau) -744)
$$
which can be checked easily by Theorem \ref{BorcherdsProduct}.  Notice that
$\Cl(K_i) =\Cl(E_i)$ is the ideal class group of $E_i$ and $j(-\bar{\tau}_i)=j(\tau_i)$. So the map $p'$ in  Lemma \ref{lem:classprojection} is an isomorphism, and
\begin{align*}
\sum_{(z_1, z_2)\in Z(W)} \log|j(z_1)-j(z_2)|
 &=4 \sum_{[\mathfrak a_i]\in \Cl(E_i)} \log|j(\tau_1^{\sigma_{\mathfrak a_1}})-j(\tau_2^{\sigma_{\mathfrak a_2}})|
 \\
  &=4\sum_{[\mathfrak a_i]\in \Cl(E_i)} \log|j(\tau_{\mathfrak a_1})-j(\tau_{\mathfrak a_2})|.
\end{align*}
Here
$$
\tau_{\mathfrak a_i} =\frac{b_i +\sqrt{d_i}}{2a_i}, \quad \ff \mathfrak a_i=[a_i, \frac{b_i +\sqrt{d_i}}{2}].
$$
So one has by Theorem \ref{theo:BigCM}
$$
-4 \sum_{[\mathfrak a_i] \in \Cl(E_i)} \log|j(\tau_{\mathfrak a_1} ) -j(\tau_{\mathfrak a_2})|^4
=C(W, K) a_1(\phi)
$$
with $\phi=\cha(\hat{\OO}_E)$, and
$$
C(W, K) = \frac{|Z(W, \sigma_2^\pm)|}{\Lambda(0, \chi_{E/F})}= \frac{2h(E_1) h(E_2)}{\Lambda(0, \chi_{E_1/\Q}) \Lambda(0, \chi_{E_2/Q})}=\frac{w_1 w_2}2,
$$
where $h(E_i)$ is the  class number of $E_i$.
By Proposition \ref{prop:Whittaker}, one has
$$
a_1(\phi) =\sum_{\substack{t \in \partial_F^{-1},\ \hbox{\tiny totally positive} \\ \tr_{F/\Q}t =1} } a(t, \phi).
$$

When  $|\hbox{Diff}(W,t) | >1$, $a(t, \phi)=0$. When $\hbox{Diff}(W,t) =\{\mathfrak p\}$, $\mathfrak p$ is inert in $E/F$ and $\ord_\mathfrak p(t\sqrt D \mathfrak p)$ is even, Proposition \ref{prop:Whittaker} implies
$$
a(t, \phi)= -4 \frac{1+\ord_\mathfrak  p (t\sqrt D)}2 \rho(t \sqrt D \mathfrak p^{-1}) \prod_{\mathfrak q <\infty} \gamma(W_\mathfrak q)\log(\norm(\mathfrak p)),
$$
since
$$
 \prod_{\mathfrak q \ne \mathfrak p} \frac{W_{t, \mathfrak q}^*(0, \phi)}{\gamma(W_\mathfrak q)}=\prod_{\mathfrak q \ne \mathfrak p}\rho_{\mathfrak q}(t\sqrt D) =\prod_{\mathfrak q} \rho_{\mathfrak q}(t\sqrt D \mathfrak p^{-1})= \rho(t \sqrt D \mathfrak p^{-1}).
$$
Here  we used the fact that $\rho_\mathfrak p(t\sqrt D \mathfrak p^{-1}) =1$  when $\mathfrak  p \in \Diff(W, t)$.
Next,  $\gamma(W_{\sigma_1}) =-i =-\gamma(W_{\sigma_2})$ implies
$$
\prod_{\mathfrak q <\infty} \gamma(W_\mathfrak q)= \prod_{\hbox{ all primes } v} \gamma(W_v)=1.
$$
So
$$
a(t, \phi) = -2 (1+\ord_\mathfrak p(t\sqrt D))  \rho(t \sqrt D\mathfrak p^{-1} ) \log (\norm(\mathfrak p)).
$$
Notice that the right hand side in the above identity is  automatically zero if we replace $\mathfrak p$ by other inert primes in $E/F$  since
 $
 \rho(t\sqrt D \mathfrak q^{-1})=0$ .
  So we have always
$$
a(t,\phi)=
   -2 \sum_{\mathfrak p \hbox{ inert in } E/F} (1+\ord_\mathfrak p(t\sqrt D))  \rho(t\mathfrak p^{-1} \partial_F) \log (\norm(\mathfrak p)).
$$
Putting everything together, and replacing $t\sqrt D$ by $t$, we obtain the theorem.

\begin{remark} \label{rem:GZformula}
It is easy to check that  our formula coincides with  \cite[(7.1)]{GZSingular}  and thus their main formula. Indeed,
\begin{equation} \label{eq4.1}
\sum_{\mathfrak a| t\OO_F} \chi_{E/F}(\mathfrak a)  \log \norm(\mathfrak a) = -\sum_{\mathfrak p \hbox{ inert in } E/F} \frac{(1+\ord_\mathfrak p(t\OO_F))}2  \rho(t\mathfrak p^{-1} ) \log (\norm(\mathfrak p))
\end{equation}
for  $t =\frac{m+\sqrt D}2\in \OO_F$ with $|m| < \sqrt D$.  To see it, for any fixed integral ideal $\mathfrak b$ of $F$, define
$$
f(\mathfrak b) = \sum_{\mathfrak a| \mathfrak b} \chi_{E/F}(\mathfrak a)  \log \norm(\mathfrak a).
$$
Write $\mathfrak b =\prod_{i=1}^n \mathfrak p_i^{e_i}$ with $e_i>0$. Assume $\mathfrak p_1$ is inert in $E/F$ and $e_1$ is odd,  write $\mathfrak b_1 =\mathfrak b \mathfrak p_1^{-e_1}$. Then (recall $\chi_{E/F}(\mathfrak p_1)=-1$).
\begin{align*}
f(\mathfrak b)
 &= \sum_{\mathfrak a_1| \mathfrak b_1}  \sum_{j=0}^{e_1} (-1)^{j} \chi_{E/F}(\mathfrak a_1) \left( j \log \norm(\mathfrak p_1)+ \log \norm(\mathfrak a_1)\right)
\\
&= \left(\sum_{j=0}^{e_1} (-1)^{j}j \log \norm(\mathfrak p_1)\right) \left(\sum_{\mathfrak a_1|\mathfrak b_1} \chi_{E/F}(\mathfrak a_1) \right)
+ \left(\sum_{j=0}^{e_1} (-1)^j\right) f(\mathfrak b_1)
\\
 &= -\frac{1+e_1}2\log \norm(\mathfrak p_1)\sum_{\mathfrak a_1|\mathfrak b_1} \chi_{E/F}(\mathfrak a_1)
 \\
 &= -\frac{1+e_1}2\log \norm(\mathfrak p_1) \prod_{i=2}^n(\sum_{j=0}^{e_i} \chi_{E/F}(\mathfrak p_i)^{j})
 \\
 &= -\frac{1+e_1}2 \rho(\mathfrak b_1)\log\norm(\mathfrak p_1)
 \\
  &=-\frac{1+e_1}2 \rho(\mathfrak b \mathfrak p_1^{-1})\log\norm(\mathfrak p_1).
\end{align*}
In particular, if there is another $\mathfrak p_i$ ($i>1$) inert in $E/F$ with $e_i$ odd, then $\rho(\mathfrak b \mathfrak p_1^{-1})=0$ and  $f(\mathfrak b)=0$.
 In our case,
$$
t \OO_F =\prod_{i=1}^n \mathfrak p_i^{e_i}.
$$
Then $\mathfrak p_i \in \Diff(W, t/\sqrt D)$ if and only if $\mathfrak p_i$ is inert in $E/F$, and $e_i$ is odd. When $$|\Diff(W, t/\sqrt D) |>1,$$ the above argument shows $f(t\OO_F)=0$, and (\ref{eq4.1}) holds as the right hand side of (\ref{eq4.1}) is also zero. When $\Diff(W, t/\sqrt D)=\{\mathfrak p\}$, say, $\mathfrak p=\mathfrak p_1$,  one has
$$
f(t\OO_F)=-\frac{1+e_1}2 \rho(\mathfrak b_1)\norm(\mathfrak p_1).
$$
The right hand side of (\ref{eq4.1}) equals this value too. So $(\ref{eq4.1})$ holds.
\end{remark}

\section{The Yui-Zagier conjecture for $\omega_i$} \label{sect:Weber}

\subsection{Borcherds product for $\omega_2(z_1) -\omega_2(z_2)$}

In this section, let
$$
L=\kzxz {\Z} {\Z} {2\Z} {\Z}  \hbox{ with } Q(X) =\det X
$$
and $\Gamma =\Gamma_0(2)$ in this section. It acts on $L'/L$ trivially,
where
$$L'/L=\left\{\mu_0=0,  \mu_1=e_{21},
\mu_2=\frac{1}2 e_{12},
\mu_3=\mu_1 + \mu_2\right\}.
$$
Here $e_{ij}$ is the $2\times 2$ matrix with the $(i,j)$ entry $1$ and all other entries $0$. It is easy to check
$Z(1, \mu_0) =X_{\Gamma_0(2)}^\Delta$ in  the open  variety $X_K=X_{\Gamma_0(2)} \times X_{\Gamma_0(2)}$.

Take the  primitive isotropic vector $\ell =-e_{12} \in L$ and the vector  $\ell'= e_{21} \in L'$ with $(\ell, \ell')=1$.  Since $(\ell, L) =2\Z$, we choose $\xi=2 \ell' \in L$ with $(\ell, \xi)=2$.
In this case,
$$
L_0'=\{ x \in L':\, (x, \ell) \equiv 0 \pmod 2\} =\{ \kzxz {a} {b/2} {2c} {d}:\, a, b, c, d \in \Z\}, \quad  L_0'/L =\{0, \mu_2\}.
$$
One has also
$$
M=L\cap (\Q \ell +\Q \ell')^\perp =\{ m(a, b) =\kzxz {a} {0} {0} {b}:\,  a, b \in \Z \},
$$
which is self-dual. So the projection $p$ from $L_0'/L$ to $M'/M$ is zero.  We further choose $\ell_M=e_{11}$ and $\ell_M'= e_{22}$ with $(\ell_M, \ell_M')=1$, so $P=0$. Finally for a weakly holomorphic modular form $f \in M_{0, \omega_L}^!$ with
$$
f(\tau) = \sum_{m, \mu} c(m, \mu) q^m \phi_\mu = \sum_{\mu} f_\mu \phi_\mu,
$$
one has
$$
f_M = f_{\mu_0}+f_{\mu_2}.
$$
 Now Theorem \ref{BorcherdsProduct} gives the following proposition in this special case.
\begin{proposition}  \label{prop5.1} Let
$$
f(\tau) =\sum_{m, \mu} c(m, \mu) q^m \phi_\mu \in M_{0, \omega_L}^!.
$$
Then there is a memomorphic modular form of two variables $\Psi(z_1, z_2, f)$ for $\Gamma_0(2) \times \Gamma_0(2)$  of parallel weight $\frac{c(0, 0)}2$ with the following product expansion near the cusp $\Q\ell$, with respect to a Weyl chamber $W$ whose closure contains $\ell_M$ ($z=(z_1, z_2)$):
$$
\Psi(z, f) = C e((\rho(W, f), z)) \prod_{\substack{ (m, n) \in \Z^2 \\ (\kzxz {-m} {0} {0} {n}, W) >0}}
 (1- q_1^{n} q_2^m)^{c(mn, 0)} (1+q_1^{n} q_2^m)^{c(mn, \mu_2)}.
$$
Here $q_j=e(z_j)$, and $|C| = 2^{\frac{c(0, \mu_2)}{2}}$.
\end{proposition}


\begin{proposition}\label{omega1} (1) \quad Let $M_{0, \omega_L}^{!, 0}$ be the subspace of  $M_{0,\omega_L}^!$ consisting of   constant vector $f=\sum a_i \phi_{\mu_i}$. Then  it is of dimension $2$ with a basis $\{ \phi_{\mu_0} + \phi_{\mu_1} ,  \phi_{\mu_0} + \phi_{\mu_2}\}$.

(2) \quad One has
\begin{align*}
\Psi(z, \phi_{\mu_0} +\phi_{\mu_1}) &=\eta(z_1) \eta(z_2),
\\
\Psi(z, \phi_{\mu_0} +\phi_{\mu_2}) &=  \sqrt 2  \eta(2z_1) \eta(2z_2),
\\
\Psi(z,  \phi_{\mu_2} -\phi_{\mu_1}) &=\frac{1}{\sqrt 2} \bold f_2(z_1) \bold f_2(z_2).
\end{align*}
Here $\bold f_2(z) =\omega_2(z)^{\frac{1}{24}}= \sqrt 2 \frac{\eta(2z)}{\eta(z)}$ is also a famous Weber function.
\end{proposition}
\begin{proof} Recall that $\SL_2(\Z)$ is generated by $n(1)=\kzxz {1} {1} {0} {1} $ and $w=\kzxz {0} {-1} {1} {0}$.
$$
n(1) (f)=a_0e_0+a_1e_1+a_2e_2-a_3e_3 =f
$$
if and only if $a_3 =0$. Next, assuming $a_3=0$, then
$$
w(f) =\sum a_i \omega_L(w)(e_i) = \frac{1}2\left[ (\sum a_i)e_0 + \sum_{i=1}^3 (a_0 + a_i -\sum_{j\ne i}a_j) e_i\right] = f
$$
if and only if $a_0=a_1+a_2$. This proves (1). In such a case, $f= a_1(e_0 + e_1) + a_2(e_0+e_2)$.

To prove (2), notice that
$$
\Gr(M)=\{ \R \kzxz {a} {0} {0} {-1}:\, a > 0 \} \cong \R_{>0}.
$$
Since $f=a_1(e_0 + e_1) + a_2(e_0+e_2) $ has no negative term,  one sees that $\Gr(M)$ has only one Weyl chamber, i.e., itself with respect to $f$. A vector $\lambda = \kzxz {-m} {0} {0} {n}$ satisfies $(\lambda, W) >0$ if and only $ m, n \ge  0$ but  not both equal to $0$. One also has
$$
\rho(W,f)=\frac{2a_2+a_1}{24} (-\ell_M+\ell_M').
$$
Now the proposition is clear from Theorem \ref{BorcherdsProduct} if we just take $C=2^{c(0, \mu_2)/2}$.
\end{proof}

\begin{proposition}\label{prop:WeberProduct} Let
$$
 f=12(\phi_{\mu_2} -\phi_{\mu_1}) + \sum_{\gamma \in \Gamma_0(2) \backslash \SL_2(\Z)} (2^{12}\omega_2^{-1}+12)|\gamma \omega_L(\gamma)^{-1}\phi_{\mu_0}  \in M_{0, \omega_L}^!
 $$
 Then
 \begin{equation} \label{eqC}
c(0, \mu_0) =c(0, \mu_1)=c(0, \mu_3)=0, \quad c(0, \mu_2)=24,
\end{equation}
 and
 $$
 \Psi(z, f) = \omega_2(z_1) -\omega_2(z_2).
 $$
 \end{proposition}
\begin{proof}
Direct calculation gives
\begin{eqnarray*}f=(q^{-1} - 98028q-  10749952q^2 -  432133182q^3+\cdots)\phi_{\mu_0}\\
                                      +(- 98296q- 10747904q^2-  432144384q^3+\cdots)\phi_{\mu_1}\\
                                     +(24- 98296q- 10747904q^2-  432144384q^3+\cdots)\phi_{\mu_2}\\
                                      +( 4096q^{\frac{1}{2}}+1228800q^{\frac{3}{2}}+74244096^{\frac{5}{2}}+\cdots)\phi_{\mu_3}.
\end{eqnarray*}
In particular, (\ref{eqC}) holds, and
 $Z(f) =Z(1, \mu_0) = X_{\Gamma_0(2)}^\Delta$ in $X_K$.   This implies that
$$
g(z_1, z_2) =\frac{\Psi(z_1, z_2, f)}{\omega_2(z_1) -\omega_2(z_2)}
$$
has no zeros  or poles in the open Shimura variety $X_K$, i.e., its divisor is supported on the boundaries  $\{P\} \times X_0(2)$ and $X_0(2) \times \{ P \}$, where $P$ runs through the cusps $0$ and $\infty$ of $X_0(2)$.  We now use Borcherds product expansion to show that $g(z_1, z_2)$ has no zeros or poles on the boundaries, and thus has to be a constant.

 The weakly holomorphic form $f$ gives rise to two Weyl chambers
 $$
\Gr(M) -Z_M(1, \mu_0) =W^\pm,
$$
where
$$
W^\pm =\{ \R \kzxz {a} {0} {0} {-1}:\, a^{\pm 1 } >1\}.
$$
We choose the Weyl chamber $W^+$ whose closure contains $\ell_M$. Then for $\lambda =\kzxz {-m} {0} {0} {n} \in M'$, $(\lambda, W^+) >0$ if and only if
$$
m+n \ge 0, \quad n \ge 0, \quad \hbox{ and } \quad  m^2+ n^2 >0.
$$
Direct calculation  using (\ref{eq:WeylVector1})--(\ref{eq:WeylVector2}) gives the Weyl vector:
$$
\rho(W^+,f)=-\frac{1}{24} (c(0,\mu_0) +c(0, \mu_2))\ell_M + (-c(-1,\mu_0)+\frac{c(0,\mu_0)+c(0,\mu_2)}{24})\ell_M'=-e_{11}.
$$
We can take the constant
$$
C =- 2^{c(0, \mu_2)/2}=-2^{12}.
$$
One has by Proposition \ref{prop5.1}
\begin{align*}
\Psi(z, f)
 &=-2^{12} q_2(1-q_1q_2^{-1})
  \prod_{m, n\ge 0, m+n >0} (1-q_1^nq_2^m)^{c(mn, 0)} (1+q_1^n q_2^m)^{c(mn, \mu_2)}
  \\
  &= 2^{12}(q_1 -q_2) \prod_{m, n>0} (1-q_1^nq_2^m)^{c(mn, 0)} (1+q_1^n q_2^m)^{c(mn, \mu_2)}.
\end{align*}
This product formula shows that $\Psi(z, f)$ has no zeros or  poles along   the boundary $\{\infty\} \times X_{\Gamma_0(2)}$ and $X_{\Gamma_0(2)} \times \{ \infty\}$. Since   $\omega_2 (z_1) -\omega_2(z_2)$ has the same property, $g(z_1, z_2)$ has no zeros or poles in these boundaries.   Fixing a $z_2 \in \H$, the function $g(z_1, z_2)$ of $z_1$ has then only zeros or poles at the cusp $\{0 \}$ in $X_0(2)$, and is thus independent of $z_1$: $g(z_1, z_2) =g(z_2)$. This implies that $g(z_2)$ has only zeros or ploes at the cups $0$, and is thus a constant $g(z_2) =A$. Therefore,
$$
\Psi(z_1, z_2, f)=A( \omega_2(z_1) -\omega_2(z_2)).
$$
Comparing the leading coefficients on both sides, one sees that $A=1$.
\end{proof}

\subsection{Proof of  Theorem \ref{theo1.3}}
Now we start to   prove  Theorem \ref{theo1.3}.
Under the isomorphism
$$
(M_2(\Q), \det ) \cong (E,  \tr_{F/\Q} \frac{x \bar x}{\sqrt D}), \quad \kzxz {x_3} {x_1} {x_4} {x_2} \mapsto \sum x_i e_i,
$$
one has
$$
L \cong \Z +  \Z \frac{D+\sqrt D}2 +\Z \frac{-d_1+\sqrt{d_1}}2 + \Z \frac{d_2 + \sqrt{d_2}}2,
$$
which is of index $2$ in $\OO_E$,  but is not an $\OO_F$-lattice unfortunately.
By Proposition \ref{prop:WeberProduct}, we have
$$
\omega_2(z_1) -\omega_2(z_2) =\Psi(z, f).
$$
 \begin{lemma} \label{lem5.4}  Assume $d_j \equiv 1 \pmod 8$. Then
 $$
 \iota_j^{-1} (K(\Gamma_0(2)) = \hat{\OO}_{E_j}^\times
 $$
 \end{lemma}
 \begin{proof} We work the case $j=2$. Then case $j=1$ is the same.  For $r= x +y \frac{d_2 +\sqrt{d_2}}2 \in E_{2, f}^\times$, one has
 $$
 \iota_2(r) = \kzxz {x+ dy } {y \frac{d-d^2}4 } {y} {x}.
 $$
 So $\iota_2(r) \in K(\Gamma_0(2))$ if and only if $y \in 2\hat\Z$. This implies
 $$
 \iota_2^{-1} (K(\Gamma_0(2)) =  (\hat\Z + 2 \hat{\OO}_{E_2})^\times.
 $$
 Since $d_2 \equiv 1 \pmod 8$, $2$ is split in $E_2$ and
 $$
 \OO_{E_2, 2}^\times =\Z_2^\times \times \Z_2^\times =(1+ 2\OO_{E_2, 2}).
 $$
 So
 $$
 (\hat\Z + 2 \hat{\OO}_{E_2})^\times = \hat{\OO}_{E_2}^\times.
 $$
 \end{proof}

 This lemma and Lemma \ref{lem3.3} implies that  the class projection $p'$ in Lemma \ref{lem:classprojection} is an isomorphism. By Proposition \ref{prop3.2}, one has
$$
\tau_j^{\sigma_{\mathfrak a_j}}=\tau_{\mathfrak a_j}=\frac{b_j +\sqrt{d_j}}{2a_j} \quad \hbox{ if }   \mathfrak a_j=[a_j, \frac{b_j +\sqrt{d_j}}{2}], 2\nmid a_j.
$$
On the other hand,
$$
\omega_2(-\bar{\tau}_j)= \omega_2(\tau_j - d_j)=\omega_2(\tau_j).
$$
So one has again by Proposition \ref{prop:BigCMPoint}
$$
\sum_{(z_1, z_2) \in Z(W)} \log|\omega_2(z_1) -\omega_2(z_2)|=  4 \sum_{[\mathfrak a_j]\in \Cl(E_j)} \log|\omega_2(\tau_{\mathfrak a_1})-\omega_2(\tau_{\mathfrak a_2})|.
$$
So we have by Theorem \ref{theo:BigCM}
$$
- 4\sum_{[\mathfrak a_j]\in \Cl(E_j)}\log| \omega_2(\tau_{\mathfrak a_1}) -\omega_2(\tau_{\mathfrak a_2})|^4
=C(W,K)[ a_1(\phi)+ 24 a_0(\tilde\phi)]  = 2 [ a_1(\phi)+24  a_0(\tilde\phi)],
$$
with $\phi=\cha(\hat L)$, and  $\tilde\phi=\cha(\mu_2 + \hat L)$. Here
$$
C(W, K) =\frac{\deg Z(W, z_{\sigma_2}^\pm)}{\Lambda(0, \chi)} =\frac{w_1w_2}2=2.
$$
Now Theorem \ref{theo1.3} follows from the following lemma which we will prove in next subsection.
\begin{lemma} \label{lem:CalWhittaker}  Let the notation be as above. Then
\begin{enumerate}
\item
$$
a_0(\tilde\phi) =0,
$$

\item
$$
a_1(\phi) =-4 \sum_{\substack{t =\frac{m+\sqrt D}2\\ |m|< \sqrt D, \hbox{\tiny odd} \\ m^2 \equiv D \pmod {16}} }
   \sum_{\mathfrak p \hbox{ inert in } E/F} \frac{1+\ord_\mathfrak p(t\OO_F)}2  \rho(t \mathfrak p^{-1}  \mathfrak p_t^{-2}) \log (\norm(\mathfrak p)).
$$
\end{enumerate}

\end{lemma}

\subsection{Whittaker functions and Proof of Lemma \ref{lem:CalWhittaker}}

\begin{lemma} \label{lem5.5} Let $W=\Q_2^2$ with the quadratic form $Q(x) =\alpha^{-1} x_1 x_2$ with $\alpha \in \Z_2^\times$. For $a=0, 1$, let
$$
M_a=\{ (x_1, x_2) \in \Z_2^2:\,  x_1 +x_2 \equiv a \pmod 2\}
$$
and
$$
\varphi_a=\cha(M_a), \quad \tilde{\varphi}_a = \cha ( (\frac{1}2, \frac{1}2) +M_a).
$$
Let $\psi$ be an unramified additive character of $\Q_2$.

\begin{enumerate}
\item When $a=0$, the local Whittaker function $W_{t\alpha}(s, \varphi_a) =0$ unless $t \in \Z_2$, and
$$
\frac{W_{t\alpha}(s,\varphi_0)}{\gamma(W)} =\begin{cases}
  \frac{1}2 &\ff t \in \Z_2^\times,
  \\
  \frac{1}2 -2^{-s} + (1-2^{-1-s}) \sum_{n=1}^{o(t)}2^{-ns} &\ff t \in 2\Z_2,
  \end{cases}
  $$
  where $o(t) =\ord_2 t$. In particular,
  $$
  \frac{W_{t\alpha}(0,\varphi_0)}{\gamma(W)}=\begin{cases}
    \frac{1}2 &\ff  o(t) =0,
    \\
    \frac{o(t)-1}{2} &\ff o(t) \ge 1.
    \end{cases}
  $$

  \item  When $a=1$, the local Whittaker function $W_{t\alpha}(s, \varphi_a) =0$ unless $t \in \Z_2$, and
  $$
  \frac{W_{t\alpha}(s,\varphi_1)}{\gamma(W)}
  =\begin{cases}
   \frac{1}2 (1-2^{-s}) &\ff t \in \Z_2^\times,
   \\
   \frac{1}2 (1+2^{-s}) &\ff t \in 2\Z_2.
   \end{cases}
  $$
  In particular,
  $$
  \frac{W_{t\alpha}(0,\varphi_1)}{\gamma(W)}
  =\begin{cases}
   0&\ff  o(t) =0,
   \\
   1 &\ff o(t) \ge 1.
   \end{cases}
  $$

\item  One has
$$
W_{t\alpha}(s, \tilde{\varphi}_a) =0
$$
unless $t -\frac{1+2a}4 \in \Z_2$, in which case it is the constant $\frac{1}2 \gamma(W)$.   In particular, $W_0(s, \tilde{\varphi}_a)=0$.

\end{enumerate}
\end{lemma}
\begin{proof} (sketch)
 By the definition and unfolding, one has
\begin{align*}
\frac{W_{t\alpha}(s, \varphi_a)}{\gamma(W) }
&=\int_{\Q_2} J_a(b) \psi(-t b) |a(wn(b))|^s \, db
\\
&=\int_{\Z_2} J_a(b) \psi(-t b)  \, db +\sum_{n \ge 1} 2^n \int_{\Z_2^\times} J_a(2^{-n} b)  \psi(-2^{-n} t b) |a(wn(2^{-n} b))|^s \, db,
\end{align*}
where
$$
J_a(b) =\int_{M_a} \psi(b x_1 x_2) \, dx_1 dx_2.
$$
Then one checks
\begin{align}
J_1(b) &=\frac{1}2 \cha(\frac{1}2\Z_2)(b),
\\
J_0(b) &=\begin{cases}
   \frac{1}2 &\ff b \in  \Z_2,
   \\
    0  &\ff b \in \frac{1}2 \Z_2^\times,
    \\
    |b|^{-1} &\ff b\notin \frac{1}2 \Z_2,
    \end{cases}
\end{align}
 and
$$
|a(wn(b))|= \min (1, |b|^{-1}).
$$
Now a direct calculation proves (1) and (2). For (3), one has similarly
\begin{align*}
\frac{W_{t\alpha}(s, \tilde{\varphi}_a)}{\gamma(W) }
&=\int_{\Q_2} \tilde{J}_a(b) \psi(-t b) |a(wn(b))|^s \, db,
\end{align*}
where
$$
\tilde{J}_a(b) =\int_{(\frac{1}2, \frac{1}2)+ M_a} \psi(b x_1 x_2) \, dx_1 dx_2 =\tilde J_a^{(0)}(b) + \tilde J_a^{(1)}(b).
$$
Here  (after a simple substitution)
\begin{align*}
4\tilde J_a^{(j)}(b)&=\int_{\Z_2} \int_{\Z_2} \psi( b(\frac{1}2+j + 2 y_1) (\frac{1}2-j +a + 2 y_2)) dy_1 \,dy_2
\\
 &= \psi((\frac{1}2+j)(\frac{1}2 -j+a) b) \cha(\Z_2)(b).
\end{align*}
So
\begin{align*}
\frac{4 W_{t\alpha}(s, \tilde{\varphi}_a)}{\gamma(W) }
&=\sum_{j=0}^1 \int_{\Z_2} \psi((\frac{1}2+j)(\frac{1}2 -j+a) b) \psi(-tb) \, db
\\
 &= 2 \int_{\Z_2} \psi((\frac{1+2a}4 -t)b) \, db
\\
 &=2 \cha(\frac{1+2a}4+\Z_2)(t).
\end{align*}
\end{proof}

To compute $a_1(\phi)$ and $a_0(\tilde\phi)$, we keep the notation in the proof of Theorem \ref{theo1.3}. Recall that
 $$
 L=\Z +  \Z \frac{D+\sqrt D}2 +\Z \frac{-d_1+\sqrt{d_1}}2 + \Z \frac{d_2 + \sqrt{d_2}}2
 $$ is not an $\OO_F$-lattice as $\frac{D+\sqrt D}2 \frac{-d_1+\sqrt{d_1}}2 \notin L$.  So $\phi$  and $\tilde\phi$ are  {\it not } factorizable over primes of $F$, instead one has  only
$$
\phi =\phi_2 \prod_{\mathfrak p \nmid 2} \phi_\mathfrak p \quad \hbox{ and }  \tilde\phi= \tilde\phi_2 \prod_{\mathfrak p \nmid 2} \phi_\mathfrak p,
$$
where $\phi_\mathfrak p= \cha(\OO_{E,\mathfrak p})$ for a prime (ideal) $\mathfrak p$ of $F$ not dividing $2$, $\phi_2=\cha(L_2)$, and $\tilde\phi_2=\cha(\frac{1}2 +L_2)$. So we need to take special care at the local calculation at $p=2$.  We focus on $\phi$  and $a_1(\phi)$ first.

 Our assumption  implies also that $2$ splits in $E$ completely. Write
$$
2\OO_F =\mathfrak p_1 \mathfrak p_2, \quad \mathfrak p_i \OO_E =\mathfrak P_i \bar{\mathfrak P}_i.
$$
Let $\sqrt D\in \Z_2$ and $\sqrt{d_i} \in \Z_2$ be some prefixed square roots of $D$ and $d_i$ respectively with $\sqrt{d_1} \sqrt{d_2} =-\sqrt D$. We identify $F_{\mathfrak p_i}$, $E_{\mathfrak P_i}$, and $E_{\bar{\mathfrak P}_i}$ with $\Q_2$ as follows.
\begin{align*}
F_{\mathfrak p_i} &\cong \Q_2, \quad \sqrt D \mapsto (-1)^{i-1}\sqrt D,
\\
E_{\mathfrak P_i}&\cong \Q_2,  \quad \sqrt D \mapsto (-1)^{i-1}\sqrt D, \sqrt{d_i} \mapsto \sqrt{d_i},
\\
 E_{\bar{\mathfrak P}_i}&\cong \Q_2,  \quad \sqrt D \mapsto (-1)^{i-1}\sqrt D, \sqrt{d_i} \mapsto -\sqrt{d_i},
\end{align*}

With this identification, we can check  that $L_2=L\otimes_\Z \Z_2$  is given by
$$
L_2=\{ x=(x_1, x_2, x_3, x_4)\in E_{\mathfrak P_1} \times E_{\bar{\mathfrak P}_1} \times E_{\mathfrak P_2} \times E_{\bar{\mathfrak P}_2}\cong \Q_2^4:\,  x_i \in \Z_2, \sum x_i \in 2\Z_2\},
$$
with quadratic form
$$
Q(x) =\frac{x_1 x_2}{\sqrt D} + \frac{x_3 x_4}{-\sqrt{D}}= Q_{\mathfrak p_1}(x_1,x_2) + Q_{\mathfrak p_2}(x_3, x_4).
$$
The embedding from $L$ to $L_2$ is given by
$$
x \mapsto (\sigma_1(x), \sigma_1(\bar x), \sigma_2(x), \sigma_2(\bar x))
$$
where $\sigma_1(\sqrt{d_i}) =\sqrt{d_i}$ and $\sigma_2(\sqrt{d_i}) =(-1)^i \sqrt{d_i}$.
So
$$
L_2 = (M_0 \times M_0) \cup  (M_1 \times M_1)
$$
where $M_a$ is given as in Lemma \ref{lem5.5}. This implies
\begin{align*}
\phi_2&=\cha(L_2) =\phi_{\mathfrak p_1, 0} \phi_{\mathfrak p_2, 0} + \phi_{\mathfrak p_1, 1} \phi_{\mathfrak p_2, 1},
\end{align*}
where
$\phi_{\mathfrak p_i, a}$ is  $\varphi_a$  in  Lemma \ref{lem5.5}.
Correspondingly, we have
$$
\phi=\phi_0  + \phi_1, \quad  a(t, \phi) = a(t, \phi_0) +a(t, \phi_1),
$$
where
$\phi_i = \phi_{\mathfrak p_1, i} \phi_{\mathfrak p_2, i}\prod_{\mathfrak p \nmid 2} \phi_\mathfrak p$. Now  Proposition \ref{prop:Whittaker} and the proof of Theorem \ref{theo:GrossZagier} give
\begin{align}
&a(t, \phi_i)  \label{eqa1}
\\
&= -4
   \sum_{\mathfrak p\ \hbox{\tiny inert in E/F}} \frac{1+\ord_\mathfrak p(t\sqrt D)}2  \rho^{(2)}(t \mathfrak p^{-1} \partial_F)
    \prod_{j=1}^2 \frac{W_{t\sqrt D, \mathfrak p_j}^{*, \psi'}(0, \phi_{\mathfrak p_j,i}) }{\gamma(W_{\mathfrak p_j})}\log (\norm(\mathfrak p)). \notag
\end{align}
Here $\psi'(x) =\psi_F(x/\sqrt D)$  and
$$
\rho^{(2)}(\mathfrak a) =\prod_{\mathfrak  p \nmid 2} \rho_\mathfrak p(\mathfrak a)
$$
as in the proof of Theorem \ref{theo:GrossZagier}.

\begin{lemma} \label{lem:Whittaker2} Assume again  $d_1 \equiv d_2 \equiv 1 \pmod 8$. Let $t \in \partial_{F}^{-1}$  with $\tr_{F/\Q}(t) =1$.  Then there is a unique prime ideal $\mathfrak p_t$ with $t\sqrt D \in \mathfrak p_t$. Moreover,
$$
W_{t\sqrt D,\mathfrak p_1}^{*, \psi'} (0, \phi_{\mathfrak p_1, 1})W_{t\sqrt D,\mathfrak p_2}^{*, \psi'}(0, \phi_{\mathfrak p_2, 1})  = 0,
$$
and
$$
\frac{W_{t\sqrt D,\mathfrak p_1}^{*, \psi'}(0, \phi_{\mathfrak p_1, 0})W_{t\sqrt D,\mathfrak p_2}^{*, \psi'}(0, \phi_{\mathfrak p_2, 0})}{\gamma(W_{\mathfrak p_1})\gamma(W_{\mathfrak p_2})}  = \ord_{\mathfrak p_t}(t\sqrt D) -1 =\rho_{\mathfrak p_t}(t\sqrt D\mathfrak p_t^{-2}).
$$
\end{lemma}
\begin{proof} Write $t =\frac{m+\sqrt D}{2\sqrt D} \in \partial_F^{-1}$. Recall the two natural embeddings $\sigma_i: F \hookrightarrow F_{\mathfrak p_i}$, $i=1, 2$.
Since
$$
\sigma_1(t\sqrt D)\sigma_2(t\sqrt D) =\frac{m^2-D}4 \equiv  0 \pmod 2,  \quad \sigma_1(t\sqrt D) - \sigma_2(t\sqrt D) \equiv 1\pmod 2,
$$
one sees that exactly one of $\ord_{\mathfrak p_i} (\sigma_i(t\sqrt D))$ is positive while the other one is zero. For simplicity, let $\mathfrak p_t=\mathfrak p_1$ with $\ord_{\mathfrak p_1}(t\sqrt D)  \ge 1$, and $\ord_{\mathfrak p_2}(t\sqrt D) =0$. Then  Lemma \ref{lem5.5} implies
$$
W_{t\sqrt D,\mathfrak p_2}^{*, \psi'}(0, \phi_{\mathfrak p_2, 1})  = 0.
$$
The same lemma also implies  (recall $L(1, \chi_{\mathfrak p_i}) =2$)
\begin{align*}
\frac{W_{t\sqrt D,\mathfrak p_1}^{*, \psi'}(0, \phi_{\mathfrak p_1, 0})W_{t\sqrt D,\mathfrak p_2}^{*, \psi'}(0, \phi_{\mathfrak p_2, 0})}{\gamma(W_{\mathfrak p_1})\gamma(W_{\mathfrak p_2})}
&= 4 \cdot \frac{1}2 \cdot \frac{1}2 (\ord_{\mathfrak p_1} (t\sqrt D) -1)
\\
&=\rho_{\mathfrak p_t}(t\sqrt D\mathfrak p_t^{-2}).
\end{align*}

\end{proof}

Now, one has by Lemma \ref{lem:Whittaker2}  and (\ref{eqa1})
\begin{align*}
a(t, \phi_1)&=0,
\\
a(t, \phi_0)&= -4
   \sum_{\mathfrak p\ \hbox{\tiny inert in E/F}} \frac{1+\ord_\mathfrak p(t\sqrt D)}2  \rho(t \mathfrak p \partial_F \mathfrak p_t^{-2})  \log (\norm(\mathfrak p)).
\end{align*}
Here $\mathfrak p_t$ is the only prime ideal of $F$ above $2$ with $t\sqrt D \in \mathfrak p_t$.
Replacing $t$ by $t/\sqrt D$, one obtains for $t =\frac{m+\sqrt D}2 \in  \OO_F$ with  $|m| < \sqrt D$
$$
a(t/\sqrt D, \phi)= -4
   \sum_{\mathfrak p\ \hbox{\tiny inert in E/F}} \frac{1+\ord_\mathfrak p(t\OO_F)}2  \rho(t \mathfrak p  \mathfrak p_t^{-2})  \log (\norm(\mathfrak p)).
$$
The condition $\rho(t \mathfrak p  \mathfrak p_t^{-2})\neq 0$ implies $t \in \mathfrak p_t^2$ and so
$$
\norm(t) =\frac{m^2-D}4 \equiv 0 \pmod 4, \quad  i.e.\  m^2 \equiv D \pmod {16}.
$$
This proves  the second identity in Lemma \ref{lem:CalWhittaker}:
$$
a_1(\phi) =\sum_{\substack{t =\frac{m+\sqrt D}2\\ |m|< \sqrt D, \hbox{\tiny odd} \\ m^2 \equiv D \pmod {16}} }
   \sum_{\mathfrak p \hbox{ inert in } E/F} \frac{1+\ord_\mathfrak p(t\OO_F)}2  \rho(t \mathfrak p^{-1}  \mathfrak p_t^{-2}) \log (\norm(\mathfrak p)).
$$

Now we prove $a_0(\tilde\phi)=0$. The same argument as above gives
$$
\tilde\phi =\tilde\phi_2 \prod_{\mathfrak p\nmid 2} \phi_\mathfrak p,
$$
and
$$
\tilde\phi_2 = \tilde\phi_{\mathfrak p_1, 0} \tilde\phi_{\mathfrak p_2, 0} + \tilde\phi_{\mathfrak p_1, 1} \tilde\phi_{\mathfrak p_2, 1}
$$
with $\tilde\phi_{\mathfrak p_i, a} $ being $\tilde\varphi_a$ in Lemma \ref{lem5.5}.  So
$$
W_{0, 2} (s,  \tilde\phi_2) = \sum_{a=0}^1  \prod_{i=0}^1 W_{0, \mathfrak p_i}(s, \tilde{\phi}_{\mathfrak p_i, a})= 0
$$
by Lemma \ref{lem5.5}. This implies
$$
W_{0, f}(s, \tilde\phi) =0
$$
and thus $a_0(\tilde\phi)=0$ by Remark \ref{rem:Constant} (and $\tilde\phi(0)=0$). This proves Lemma \ref{lem:CalWhittaker}, and thus Theorem  \ref{theo1.3}.

\begin{remark} When $d_i \equiv 1 \pmod 8$ are not satisfied, the big CM value formula will still give a factorization formula for the CM values of $\omega_2(z_1) -\omega_2(z_2)$ although the summation will be over the ring class group of $E_i$ with conductor $2$ when $d_i \equiv 5 \pmod 8$ (see Lemma \ref{lem5.4}). We leave the details to the reader.

\end{remark}

\begin{remark} The Weber function $\omega_2$ has two companions $\omega_1(\tau)=w(\omega_2)$ and $\omega_0(\tau) =\omega_1(\tau+1)$. So the results on $\omega_2$ can easily be transferred to its companions $\omega_0$ and $\omega_1$.
\end{remark}

\bibliographystyle{alpha}
\bibliography{reference}

\end{document}